\documentclass[11pt,a4paper]{article}
\usepackage[]{geometry}
\usepackage[mathcal]{euscript}
\usepackage{bm,url}                     
\usepackage{amsfonts}
\usepackage{amssymb}
\usepackage{amsmath}
\usepackage{amsthm}
\usepackage[utf8x]{inputenc}
\usepackage{float}
\usepackage[pdftex]{graphicx}
\DeclareGraphicsExtensions{.bmp,.png,.pdf,.jpg,.ps}
\usepackage{colortbl}  
\usepackage{todonotes}             
\usepackage{booktabs}
\newtheorem{definition}{Definition}
\newtheorem{remark}{Remark}
\newtheorem{theorem}{Theorem}
\newtheorem{lemma}{Lemma}
\newtheorem{proposition}{Proposition}
\newtheorem{corollary}{Corollary}

\newtheorem{notation}[theorem]{Notation}

\usepackage[english]{babel}
\usepackage{amsmath}
\usepackage{amsfonts}
\usepackage{amssymb}
\usepackage{graphicx}
\usepackage{color}
\usepackage{array}
\usepackage{mathrsfs}
\usepackage{hyperref}

\definecolor{violet}{rgb}{0.7,0,0.6}
\definecolor{OliveGreen}{RGB}{85,107,47}

\title{Modelling and Parameter Estimation for Discretely Observed Fractional Iterated Ornstein--Uhlenbeck Processes}
\author{Juan Kalemkerian\\
Universidad de la República, Facultad de Ciencias.}

\begin{document}
\maketitle
\begin{abstract}
\noindent 
In this work we present how to model an observed time series by a FOU$(p)$ process.
We will show that the FOU$(p)$ processes can be used  to model a wide range of time series
varying from short range dependence to long range dependence, with performance similar 
to the ARMA or ARFIMA models and in several cases outperforming them. Also, we
extend the theoretical results for any FOU$(p)$ processes 
for the case in which the Hurst parameter is less than $1/2$ and 
we show theoretically and by simulations that under some conditions on $T$ and the sample
size $n$ it is possible  to obtain consistent  estimators of  the parameters  
when the process is observed 
in a discretized and equispaced  interval $[0,T]$.
Lastly, we  give a way to obtain explicit formulas for the auto-covariance function for any 
 FOU$(p)$
 and  we present an application for FOU$(2)$ and FOU$(3)$.

\end{abstract}

AMS: 62M10

\newpage

\section{Introduction}
Frequently, the real time series data sets that can be found in the applications are measurements
of a certain variable at equispaced intervals. The nature of many of these processes is in continuous 
time. Although there are many continuous time stochastic processes  that can be used to model
these situations, the  discrete time models such as ARMA or ARFIMA remain the most popular  
for practitioners.
In \cite{chichi}, the continuous time FOU$(p)$ processes are defined. The FOU$(p)$ are centred Gaussian
 stationary processes and are a particular case of more general processes defined in \cite{Arratia}
 when 
the functional defined in \cite{Arratia} are applied to a fractional Brownian motion. 
A FOU$(p)$ process has two parameters, $H$ and $\sigma$, given by the fractional Brownian motion, and 
in \cite{chichi} it is proved that $H$ gives information
about the irregularity of the trajectories, because it is proved that (using a result of 
\cite{Ibragimov}
about the relation between the variogram and the H\"older index of any Gaussian process) $H$ is the
H\"older index of
any FOU$(p)$. Then it is possible to apply a procedure suggested in \cite{Istas} to estimate $H$ and 
$\sigma$ in a consistent way. The FOU$(p)$ processes also contain other parameters that give
information about the local dependence, which we  called the $\lambda$ parameters. 
 In \cite{chichi} the theoretical properties are 
established and a method to estimate their parameters with their asymptotic behaviour
was constructed. 
The estimation method and the asymptotic results for the $\lambda$ parameters were obtained under the assumption that the process is observed in the entire interval 
$[0,T]$ where $T$ goes to $\infty$. This condition is unrealistic because in practice every
sample has a finite number of observations. This difficulty will be resolved in the present paper.
For $p\geq 2$, the FOU$(p)$ processes have short range dependence, 
and when $p=1$ we have the fractional Ornstein--Uhlenbeck processes (FOU) defined in \cite{Cheridito}, which
have long range dependence when $H>1/2$. Also, the FOU$(p)$ processes have  a  continuity
structure in the $\lambda$ parameters, which allows us  to approximate an FOU process by a subfamily of FOU$(2)$. 
Thus, the FOU$(p)$ processes can be viewed as a generalization of FOU processes and can be used to model
both types of time series: short and long range dependence.
In the present paper we will present a consistent way to estimate the $\lambda$ parameters when 
the process is viewed for $n$ equispaced observations within the interval $[0,T]$.
In Section 2,  we give the definition of an FOU$\left(
p\right) $ process, as defined in \cite{chichi}, and summarize the theoretical properties that established in \cite{chichi}
for $H>1/2$. 
Also we will present a way to extend those properties to the case $H<1/2$.
In Section 3, we give a procedure to estimate all parameters at once in a consistent way, when the 
process is observed on an equispaced sample in $[0,T]$. Also we include explicit formulas 
for the auto-covariance function of any FOU$(2)$ or FOU$(3)$ process, and present a way to obtain 
similar formulas for any FOU$(p)$. In Section 4, we 
corroborate the theoretical results by simulations. We perform a
small simulation study for FOU$(2)$ processes, give the estimations of the parameters and the standard deviations of all of them, for different sample sizes, different values of $T$, and for different values of the true
parameters. In Section 5, we show how we can model an observed time series by a FOU process.
In Section 6 we apply the FOU$(p)$ processes to model  three real data sets, 
two with short range dependence and the other with long range dependence, 
and compare the performance of these
models with ARMA or ARFIMA models according to their predictive power. In Section 7 we make some final 
remarks. Our concluding remarks are given in Section 8.
In Section 9, we present 
the proof of the results established in Section 3.

\section{Definitions and properties}
We start with the definition of a  fractional Brownian motion and iterated Ornstein--Uhlenbeck process.

\begin{definition}
	A fractional Brownian motion with Hurst parameter $H\in \left( 0,1\right] $,
	is an almost surely continuous centred Gaussian process $\left\{
	B_{H}(t)\right\} _{t\in \mathbb{R}}$ such that its auto-covariance function is
	\begin{equation*}
	\mathbb{E}\left( B_{H}(t)B_{H}(s)\right) =\frac{1}{2}\left( \left\vert
	t\right\vert ^{2H}+\left\vert s\right\vert ^{2H}-\left\vert t-s\right\vert
	^{2H}\right) ,\text{ \ }t,s\in \mathbb{R}.
	\end{equation*}
\end{definition}

We next give the definition of a fractional iterated Ornstein--Uhlenbeck processes
of order $p$ (FOU$(p)$), as defined in \cite{chichi}. 

\begin{definition}
	Suppose that $\left\{ \sigma B_{H}(s)\right\} _{s\in \mathbb{R}}$ is a fractional
	Brownian motion with Hurst parameter $H$ and scale parameter $\sigma$.
	Suppose further that $\lambda _{1},\lambda
	_{2},...,\lambda _{q}$ are distinct positive numbers and that $p_{1},p_{2},...,p_{q}\in \mathbb{N}$ are such that $%
	p_{1}+p_{2}+...+p_{q}=p$. Then a fractional iterated Ornstein--Uhlenbeck process
	of order $p$ is any process of the form
	\begin{equation*}
	X_{t}:= T_{\lambda _{1}}^{p_{1}} \circ T_{\lambda _{2}}^{p_{2}} \circ ....
	\circ T_{\lambda _{q}}^{p_{q}} (\sigma
	B_{H})(t)=\sum_{i=1}^{q}K_{i}\left( \lambda \right)
	\sum_{j=0}^{p_{i}-1}\binom{p_{i}-1}{j} T_{\lambda _{i}}^{\left(
		j\right) }(\sigma B_{H})(t),
	\end{equation*}%
	where the numbers $K_{i}\left( \lambda \right) $  are defined by \begin{equation}
K_{i}\left( \lambda \right) =K_{i}\left( \lambda _{1},\lambda
_{2},...,\lambda _{q}\right): =\frac{1}{\prod\limits_{j\neq i}\left(
	1-\lambda _{j}/\lambda _{i}\right) }  \label{k_i}
\end{equation} and the operators $%
	T_{\lambda _{i}}^{\left( j\right) }$  satisfy
	\begin{equation}
T_{\lambda }^{(h)}(y)(t):=\int_{-\infty }^{t}e^{-\lambda (t-s)}\frac{\left(
	-\lambda \left( t-s\right) \right) ^{h}}{h!}dy(s) \ \ \text {for}  \ \ h=0,1,2,... \label{hh}
\end{equation}
When $h=0$ we simply call this $T_\lambda$, thus \begin{equation}
T_{\lambda }(y)(t):=\int_{-\infty }^{t}e^{-\lambda (t-s)} dy(s). \label{h=0}
\end{equation}
\label{fou_definition}
\end{definition}
\begin{remark}
 The equality between $T_{\lambda _{1}}^{p_{1}} \circ T_{\lambda _{2}}^{p_{2}} \circ ....
	\circ T_{\lambda _{q}}^{p_{q}}$ and $\sum_{i=1}^{q}K_{i}\left( \lambda \right)
	\sum_{j=0}^{p_{i}-1}\binom{p_{i}-1}{j} T_{\lambda _{i}}^{\left(
		j\right) } $ is proved in \cite{Arratia}.
\end{remark}

\begin{remark}
	Observe that the composition $T_{\lambda _{1}}^{p_{1}} \circ T_{\lambda _{2}}^{p_{2}} \circ ....
	\circ T_{\lambda _{q}}^{p_{q}} $ given in Definition \ref{fou_definition} is commutative. 
	Then, to avoid ambiguity in the estimation of the $\lambda$ we will assume that $\lambda_1<\lambda_2<...<\lambda_q$.
\end{remark}
\begin{notation}
	$\left\{ X_{t}\right\} _{t\in \mathbb{R}}\sim  \text{FOU } \left( \lambda _{1}^{\left(
		p_{1}\right) },\lambda _{2}^{\left( p_{2}\right) },...,\lambda _{q}^{\left(
		p_{q}\right) },\sigma ,H\right), $ where $0<\lambda_1<\lambda_2<...<\lambda_q$ or more simply,
	$\left\{ X_{t}\right\} _{t\in \mathbb{R}}\sim$FOU$(p)$.
	
\end{notation}

Observe that the notation FOU$\left( \lambda _{1}^{\left( p_{1}\right)
},\lambda _{2}^{\left( p_{2}\right) },...,\lambda _{q}^{\left( p_{q}\right)
},\sigma ,H\right) $ implies that the parameters $\lambda _{i}$ are 
distinct. Also, the notation FOU$(p)$ means that we have taken the composition of $T_\lambda$ $p$ times.

\begin{remark}
	When $p_{1}=p_{2}=...=p_{q}=1$ \ the process is equal to
\end{remark}

\begin{equation}\label{lambdas distintos}
X_{t}=T_{\lambda _{1}} \circ T_{\lambda_2} \circ...\circ T_{\lambda_q}(\sigma
B_{H})(t)=\sum_{i=1}^{q}K_{i}\left( \lambda \right) T_{\lambda _{i}}(\sigma
B_{H})(t)
\end{equation}%
and we write $\left\{ X_{t}\right\} _{t\in \mathbb{R}}\sim $FOU$\left(
\lambda _{1},\lambda _{2},...,\lambda _{q},\sigma ,H\right) .$

\begin{remark}
	When $p=1$, we obtain a fractional Ornstein--Uhlenbeck process (FOU$\left(
	\lambda ,\sigma ,H\right) $).
\end{remark}

\begin{remark}
	Any FOU$\left( \lambda _{1}^{\left( p_{1}\right) },\lambda _{2}^{\left(
		p_{2}\right) },...,\lambda _{q}^{\left( p_{q}\right) },\sigma ,H\right) $,
	is a Gaussian, centred, and almost surely continuous process.
\end{remark}
Any FOU$(p)$ has the property that almost all its trajectories are everywhere non differentiable. This fact 
will be used in Section 3 to obtain  estimators of $H$ and $\sigma$.
The auto-covariance function of any $FOU\left( \lambda _{1},\lambda
_{2},...,\lambda _{p},\sigma ,H\right) $  is 
\begin{equation}
\mathbb{E}\left( X_{0}X_{t}\right) =\frac{\sigma ^{2}H}{2}\sum_{i=1}^{p}%
\frac{\lambda _{i}^{2p-2H-2}}{\prod\limits_{j\neq i}\left( \lambda
	_{i}^{2}-\lambda _{j}^{2}\right) }f_{H}(\lambda _{i}t)  \label{covfoup}
\end{equation} where $p \geq 2$ and the function $f_H$ is defined by 
\begin{equation}
f_{H}(x):=e^{-x}\left( \Gamma \left( 2H\right)
-\int_{0}^{x}e^{s}s^{2H-1}ds\right) +e^{x}\left( \Gamma \left( 2H\right)
-\int_{0}^{x}e^{-s}s^{2H-1}ds\right) .  \label{fH}
\end{equation}

It is known that for $H>1/2$, every FOU$\left( \lambda ,\sigma ,H\right) $ is a long
memory process, Cheridito et al. \cite{Cheridito}, that is 
$\sum_{n=-\infty }^{+\infty }\left\vert \gamma
\left( n\right) \right\vert =+\infty $ where $\gamma \left( n\right) =%
\mathbb{E}\left( X_{0}X_{n}\right) .$ In \cite{chichi} it is proved that if we
compose at least two operators of the form $T_{\lambda }$ evaluated for a
fractional Brownian motion, with Hurst parameter $H>1/2$, we obtain a
process $\left\{ X_{t}\right\} _{t\in \mathbb{R}}$ that satisfies $%
\sum_{n=-\infty }^{+\infty }\left\vert \mathbb{E}\left( X_{0}X_{n}\right)
\right\vert$ $<+\infty.$  Further, any FOU$(p)$ with $p \geq 2$ 
has short memory. Therefore, FOU$(p)$  has a short memory for $p \geq 2$ and a long memory
for $p=1$. Also, any FOU$(\lambda_1,\lambda_2,\sigma,H)$ goes to some FOU$(\lambda_2,\sigma,H)$ when
$\lambda_1$ goes to zero. Then, the FOU$(\lambda_1,\lambda_2,\sigma,H)$ for small values of $\lambda_1$ can be
used to model both: a short range dependence and a long range dependence.

Observe that the  $f_{H}$ are well defined for all $H\in \left[ 0,1%
\right] .$ Pipiras and Taqqu (\cite{Pipiras}) proved that when $H>1/2,$   
\begin{equation}
\mathbb{E}\left( \int \int_{\mathbb{R}^{2}}f(u)g(v)dB_{H}(u)dB_{H}(v)\right)
=H\left( 2H-1\right) \int \int_{\mathbb{R}^{2}}f(u)g(v)\left\vert
u-v\right\vert ^{2H-2}dudv \label{pipiras_taqqu_equlity}
\end{equation}%
holds for every $f$ and $g$ such that \begin{equation}
                                 \int \int_{\mathbb{R}^{2}}\left\vert
f(u)g(v)\right\vert \left\vert u-v\right\vert ^{2H-2}dudv<+\infty .
\label{fg}
                                \end{equation}

It is well known that (\ref{pipiras_taqqu_equlity}) does not hold for $H \leq 1/2$  for every $f,g$ 
such that (\ref{fg}) holds. Neverthless, 
Cheridito et al. (\cite{Cheridito}) have proved that  the last  equality  
remains valid for the exponential
functions ($f$ and $g$) that appear in the fractional iterated Ornstein--Uhlenbeck
processes for values of $H\in \left( 0,1/2\right) $ too.
Therefore, we can follow the same line of proof as in \cite{chichi} to
prove that 
all  the theoretical results obtained in Section 2 of \cite{chichi} remain valid
for all $H\in \left( 0,1/2\right) \cup \left( 1/2,1\right) $. 

\section{Parameter estimation}
 Section 3 of \cite{chichi} presents a procedure that allows estimating the parameters of
any FOU$\left( p\right) $ in a consistent way. As with 
estimators $(\lambda,\sigma,H)$ proposed in \cite{Brouste} for the fractional Ornstein--Uhlenbeck process, 
this procedure has two steps. 
Firstly, we can
estimate $\sigma $ and $H$ independently of the values of the $\lambda _{i}$.
 Secondly, taking advantage for the explicit
formula of the spectral density, and using $(\widehat{H},\widehat{\sigma})$ instead of $(H,\sigma)$, we can 
estimate the $\lambda$
using Whittle estimators. To estimate $H$ and $\sigma$ it is enough to have an equispaced 
sample of $[0,T]$ and the results for consistency and asymptotic normality are valid for $H>1/2$.
To estimate the $\lambda$ it is necessary to observe the process over the whole 
interval $[0,T]$.

In this section we present a way to extend the theoretical results to estimate $H$ and $\sigma$
for $H<1/2$ and we will show a procedure to consistently estimate  the $\lambda$ parameters
when the process is observed on an equispaced sample of $[0.T].$

\subsection{Estimation of $H$ and $\sigma$ }

We start defining by filter of length $k+1$ and order $L$.
\begin{definition}
 $a=\left( a_{0},a_{1},...,a_{k}\right) $\ is  a \textit{filter of length $k+1$\ and order $L\geq 1$}\ if and only if the following 
conditions hold:

\begin{itemize}
	\item $\sum_{i=0}^{k}a_{i}i^{l}=0$\ para todo $0\leq l\leq L-1.$
	
	\item $\sum_{i=0}^{k}a_{i}i^{L}\neq 0.$
	
\end{itemize}
\end{definition}

Observe that given $a,$ a filter of order $L$\ and length $k+1$, the new filter that we call $a^2$ and is defined by
$a^{2}=\left( a_{0},0,a_{1},0,a_{2},0,...0,a_{k}\right) $ has order $L$\ and length $2k+1$. 
Now, we define the quadratic variation of a sample associated to a filter $a$ as follows.
\begin{definition}
	Given a filter $a$\ of length $k+1$ and a sample $X_{1},X_{2},...,X_{n}$, we define 
	\begin{equation*}
	V_{n,a}:=\frac{1}{n}\sum_{i=0}^{n-k}\left( \sum_{j=0}^{k}a_{j}X_{i+j}\right)
	^{2}.
	\end{equation*}
\end{definition}

The following theorem defines $\left ( \hat{H}, \hat{\sigma} \right )$ and summarizes their asymptotic properties. 

\begin{theorem}[Kalemkerian \& Le\'on]\label{asymptotic of H sigma}\[\]
	If $X_{\Delta },X_{2\Delta },....,X_{i\Delta },...,X_{n\Delta }=X_{T}$ is an equispaced sample of the process $\left\{ X_{t}\right\} _{t\in 
		\mathbb{R}}\sim $FOU$\left(p\right) $ where $H>1/2$, the filter $a$\ is of order $L\geq 2$ and length $k+1$, $\Delta_n=n^{-\alpha}$ for some 
		$\alpha$ such that $0<\alpha<\frac{1}{2(2H-1)}$ and $T=n\Delta_n \rightarrow +\infty$,
	as $n\rightarrow +\infty. $ Define 
	\begin{equation}
\widehat{H}=\frac{1}{2}\log _{2}\left( \frac{V_{n,a^{2}}}{V_{n,a}}\right), 
\label{Hgorro}
\end{equation}
\begin{equation}
\widehat{\sigma }=\left( \frac{-2V_{n,a}}{\Delta _{n}^{2\widehat{H}%
	}\sum_{i=0}^{k}\sum_{j=0}^{k}a_{i}a_{j}\left\vert i-j\right\vert ^{2\widehat{%
			H}}}\right) ^{1/2}. \label{sigmagorro}
\end{equation}
Then

	\begin{enumerate}
		\item 
		\begin{equation*}
		\left( \widehat{H},\widehat{\sigma }\right) \overset{c.s.}{\rightarrow }%
		\left( H,\sigma \right) .
		\end{equation*}
		
		\item 
		\begin{equation*}
		\sqrt{n}\left( \widehat{H}-H\right) \overset{w}{\rightarrow }N\left(
		0,\Gamma _{1}\left( H,\sigma ,a\right) \right)
		\end{equation*}
		
		\item 
		\begin{equation*}
		\frac{\sqrt{n}}{\log n}\left( \widehat{\sigma }-\sigma \right) \overset{w}{%
			\rightarrow }N\left( 0,\Gamma _{2}\left( H,\sigma ,a\right) \right)
		\end{equation*}
	\end{enumerate}
\end{theorem}

\begin{remark}
 Clearly the hypothesis $0<\alpha <\dfrac{1}{2(2H-1)}$ in Theorem \ref{asymptotic of H sigma} 
 needs the  condition that $H>1/2$, but when $H<1/2$, we can use Theorem 3 (iii) 
 of Istas \& Lang (\cite{Istas}) in the case $s<1$ and follow the same 
 proof taking $\alpha >1/2$.
\end{remark}

\subsection{Estimation of the  $\lambda$ parameters}
\noindent   If $X=\{X_t\}_{t\in\mathbb{R}}\sim  $ 
$ FOU(\lambda^{(p_1)}_1,\ldots,\lambda^{(p_q)}_q,\sigma,H)$ where $\sum_{i=1}^qp_i=p.$
It is proved in \cite{chichi} that the spectral density of $X$ is

\begin{equation}
f^{(X)}(x)=\frac{\sigma^2\Gamma(2H+1)\sin(H\pi)|x|^{2p-1-2H}}{2\pi\prod_{i=1}^q(\lambda^2_
	i+x^2)^{p_i}}. \label{espectral}
\end{equation}

\noindent If $H$ and $\sigma$ are known, taking advantage of the 
explicit knowledge of the spectral density and if the process is observed completely on $[0,T]$, we can proceed as in 
\cite{Leo} to estimate the rest of the parameters by using a modified  Whittle contrast. 
 
\begin{theorem}[Kalemkerian \& Le\'on]\label{asymptotic lambdas}\[\]
	Suppose given $\left\{ X_{t}\right\} _{t\in \mathbb{R}}\sim FOU\left( \lambda
	_{1}^{\left( p_{1}\right) },\lambda _{2}^{\left( p_{2}\right) },...,\lambda
	_{q}^{\left( p_{q}\right) },\sigma ,H\right) $ where $\sigma $ and $H$ are
	known. Suppose further that the true value of the parameter is $\lambda ^{0}=\left( \lambda _{1}^{0},\lambda
	_{2}^{0},...,\lambda _{q}^{0}\right) $\ $\in $int$( \Lambda)$ where $\Lambda \subset 
	\left \{ \lambda \in 
	 \mathbb{R}^{q} \  : \  0<\lambda_1 <\lambda_2<...<\lambda_q \right \}$ is compact and 
	 the process is observed on 
	$\left[ 0,T\right] $ for some $T>0.$ Define the following contrast process: 
	\begin{equation*}
	U_{T}\left( \lambda \right) =\frac{1}{4\pi }\int_{-\infty }^{+\infty }\left(
	\log f^{\left( X\right) }\left( x,\lambda \right) +\frac{I_{T}\left(
		x\right) }{f^{\left( X\right) }\left( x,\lambda \right) }\right) w\left(
	x\right) dx
	\end{equation*}%
	where $f^{\left( X\right) }\left( x,\lambda \right) $ is the spectral
	density of the process given in (\ref{espectral}), $I_{T}\left( x\right) $ is the
	periodogram of the second order 
	\begin{equation*}
	I_{T}\left( x\right) =\frac{1}{2\pi T}\left\vert
	\int_{0}^{T}X_{t}e^{-itx}dt\right\vert ^{2}
	\end{equation*}%
	and $w(x)=\frac{\left\vert x\right\vert }{1+\left\vert x\right\vert ^{b}}$
	where  $b>2.$ Then $\widehat{\lambda 
	}_{T}=\arg \min_{\lambda \in \Lambda }U_{T}\left( \lambda \right) $ satisfies
	
	\begin{itemize}
		\item $\widehat{\lambda }_{T}\overset{P}{\rightarrow }\lambda ^{0}$ when $%
		T\rightarrow +\infty $ and
		
		\item $\sqrt{T}\left( \widehat{\lambda }_{T}-\lambda ^{0}\right) \overset{D}{%
			\rightarrow }N_{q}\left( 0,W_{1}^{-1}\left( \lambda ^{0}\right)
		W_{2}^{{}}\left( \lambda ^{0}\right) W_{1}^{-1}\left( \lambda ^{0}\right)
		\right) $ when $T\rightarrow +\infty $
	\end{itemize}
	
	\noindent where $N_{q}(.,.)$ denotes the $q-$dimensional Gaussian law and the matrices 
	$W_{1}\left( \lambda ^{0}\right) $ and $W_{2}\left( \lambda ^{0}\right) $
	are defined by
	$$W_{1}\left( \lambda \right) =\left( w_{ij}^{\left( 1\right)
	}\left( \lambda \right) \right) _{i,j=1,...,q} \  \mbox{and} \ \  W_{2}\left( \lambda
	\right) =\left( w_{ij}^{\left( 2\right) }\left( \lambda \right) \right)
	_{i,j=1,...,q}$$ where
	\begin{eqnarray*}
		w_{ij}^{\left( 1\right) }\left( \lambda \right)  &=&\frac{1}{4\pi }%
		\int_{-\infty }^{+\infty }w(x)\frac{\partial }{\partial \lambda _{i}}\log
		f^{\left( X\right) }\left( x,\lambda \right) \frac{\partial }{\partial
			\lambda _{j}}\log f^{\left( X\right) }\left( x,\lambda \right) dx \\
		w_{ij}^{\left( 2\right) }\left( \lambda \right)  &=&\frac{1}{4\pi }%
		\int_{-\infty }^{+\infty }w^{2}(x)\frac{\partial }{\partial \lambda _{i}}%
		\log f^{\left( X\right) }\left( x,\lambda \right) \frac{\partial }{\partial
			\lambda _{j}}\log f^{\left( X\right) }\left( x,\lambda \right) dx.
	\end{eqnarray*}
	
\end{theorem}

In the following theorem, we show the main theoretical result of this work that we will show that works well in the simulations (Section 4), and we use in the applications to real data (Section 6), that is, it is possible to take a discretized version of $U_T$ 
and $I_T$, and  using $\left( \widehat{H},\widehat{\sigma }\right)$ (given in
(\ref{sigmagorro}) and (\ref{Hgorro})) instead of  $(\sigma^0, H^0)$, we obtain the same
consistency result at the cost to adding a hypothesis about the speed with which $T_n$ tends to 
infinity and the need to change the function $w$. In this way, we can estimate all the parameters consistenly in a FOU$(p)$ process observed in an equiespacied sample of $[0,T].$

\begin{theorem}
Suppose $X_{\Delta },X_{2\Delta },X_{3\Delta },...,X_{n\Delta }$ is an equispaced
sample in $\left[ 0;T\right] $ of some $\left\{ X_{t}\right\} _{t\in \mathbb{R}%
}\sim $ \\ FOU$\left( \lambda _{1}^{\left( p_{1}\right) },...,\lambda
_{q}^{\left( p_{q}\right) },\sigma ,H\right) $ where $%
p_{1}+p_{2}+...+p_{q}=p $. Suppose further that $\left( \lambda ,\sigma ,H\right)
\in \Lambda \times \left[ \sigma_1,\sigma_2\right] \times \left[ h_1,h_2\right] $ where $%
\sigma_1>0, $ $0<h_1<h_2<1,$ and $\Lambda \subset \left \{\lambda \in \mathbb{R}^q \ : \ 0<\lambda_1 < \lambda_2<...<\lambda_q \right \} $ is
compact. We call $\left( \lambda ^{0},\sigma ^{0},H^{0}\right) \in \overset{o%
}{\Lambda }\times \left( \sigma_1,\sigma_2\right) \times \left( h_1,h_2\right) $  the real
vector of parameters. Define the weight function $w(x)=\frac{|x|^a}{1+|x|^b}$ where $a\geq 2p$ and $b \geq a+3$.

Define the functions (for any fixed $T>0$) 
\begin{eqnarray*}
	U_{T}\left( \lambda ,\sigma ,H\right) &=&\int_{0}^{T}h_{T}(x,\lambda ,\sigma
	,H)dx\text{ and } \\
	U_{T}^{\left( n\right) }\left( \lambda ,\sigma ,H\right) &=&\frac{T}{n}%
	\sum_{i=1}^{n}h_{T}^{\left( n\right) }\left( iT/n,\lambda ,\sigma ,H\right)
\end{eqnarray*}%
where the functions $h_{T}$ and $h_{T}^{\left( n\right) }$ are defined by 
\begin{equation*}
h_{T}(x,\lambda ,\sigma ,H)=\frac{1}{2\pi }\left( \log f^{\left( X\right)
}\left( x,\lambda ,\sigma ,H\right) +\frac{I_{T}(x)}{f^{\left( X\right)
	}\left( x,\lambda ,\sigma ,H\right) }\right) w(x),
\end{equation*}%
\begin{equation*}
h_{T}^{\left( n\right) }(x,\lambda ,\sigma ,H)=\frac{1}{2\pi }\left( \log
f^{\left( X\right) }\left( x,\lambda ,\sigma ,H\right) +\frac{I_{T}^{\left(
		n\right) }(x)}{f^{\left( X\right) }\left( x,\lambda ,\sigma ,H\right) }%
\right) w(x)
\end{equation*}%
where 
\begin{equation*}
I_{T}(x)=\frac{1}{2\pi T}\left\vert \int_{0}^{T}e^{itx}X_{t}dt\right\vert
^{2}\text{ and }I_{T}^{\left( n\right) }(x)=\frac{T}{2\pi }\left\vert \frac{1%
}{n}\sum_{j=1}^{n}e^{\frac{ijTx}{n}}X_{\frac{jT}{n}}\right\vert ^{2}
\end{equation*}%
are the periodogram and the discretization of the periodogram respectively.

Define 
\begin{eqnarray}
	\widehat{\lambda }_{T} &=&\arg \min_{\lambda \in \Lambda }U_{T}\left(
	\lambda ,\sigma ^{0},H^{0}\right),\\
	\widehat{\lambda }_{T}^{\left( n\right) } &=&\arg \min_{\lambda \in \Lambda
	}U_{T}^{\left( n\right) }\left( \lambda ,\widehat{\sigma },\widehat{H}\right) 
	\label{lambdasgorro}
\end{eqnarray}%
where $\widehat{\sigma }$ and $\widehat{H}$ are defined by (\ref{sigmagorro}) 
and (\ref{Hgorro}), respectively.

Suppose that $\ $the minimum of $U_{T}\left( \lambda ,\sigma
^{0},H^{0}\right) $ is reached at a unique point $\widehat{\lambda }_{T}.$
\begin{itemize}
 \item[(A)] If $1/2<H<5/6$ , and if $T_{n}=n^{1-\alpha }$  where 
 $\frac{3}{4}<\alpha <\min \{\frac{1}{2\left( 2H-1\right) }, 1 \}.$ 
\item[(B)] If $H<1/2$ , and if $T_{n}=n^{1-\alpha }$  where 
$\max \{ \frac{1}{H+1}, \frac{3}{4} \}  <\alpha <1. 
$ 
\end{itemize}

 Then 
\begin{equation*}
\lim_{n\rightarrow +\infty }\widehat{\lambda }_{T_{n}}^{\left( n\right) }%
\overset{P}{=}\lambda ^{0}.
\end{equation*}\label{consistency_theorem}
\end{theorem}

\begin{remark}\label{T/n}
 Conditions over $\alpha$ given in $(A)$ and $(B)$, allows to affirm that 
 as $n \rightarrow +\infty$,
 $\frac{T_n^{4}}{n} \rightarrow 0$ and  $\frac{T_n^{H+1}}{n^{H}} \rightarrow 0.$
\end{remark}
\begin{remark}
Results concerning to the convergence for the fractional Ornstein--Uhlenbeck processes are given for 
$1/2<H<3/4$ as can bee seen for example in \cite{Brouste} and \cite{Xiao}. In our case we have extend
the theorem of convergence to $1/2<H<5/6.$
\end{remark}

\begin{corollary}
 If $H$ and $\sigma$ are known, the result established in Theorem \ref{consistency_theorem}
 is still valid, just changing the definition of $\widehat{\lambda}_T^{(n)}$ given in 
 (\ref{lambdasgorro}) to 
 $\widehat{\lambda}_T^{(n)}=argmin_{\lambda \in \Lambda} U_{T}^{(n)}\left (\lambda,\sigma^{0},H^{0} 
 \right ).$\label{estimo_solo_lambda}
\end{corollary}

\begin{remark}
 The estimation of $H$ does not depend on the selection of $T$, but the estimation
 of $\sigma$ depends on $T$. Also, in the real data set considered in Section 5,
 we will see that  $\hat{\sigma}$ varies considerably as a function of $T$.
 The presence of the parameter $\sigma$ in the FOU$(p)$ model is simply 
 as a  multiplicative factor in the auto-covariance function. We show in Section 5 that
 we can choose previously  a value of $\sigma$ (for example $\sigma=1$),
 and consider the FOU$(p)$ process as a model with parameters $H$ and $\lambda$, and 
 all the theoretical results about $\hat{H}$ and $\hat{\lambda}$ remain valid.
 
\end{remark}

\begin{remark}
	The study of the asymptotic distribution of the  
	estimator of $\lambda $ is left for future work.
	It may be enough to find the relation between $T_n$ and $n$ to obtain the asymptotic 
	normality, or maybe it will be necessary to take the 
	observations at random points in the interval $[0,T]$, as can be seen in 
	\cite{Masry} and \cite{Bardet}.
\end{remark}

Being the FOU$(p)$ a Gaussian process, to give a complete description of FOU$(2)$ and FOU$(3)$ processes, in the following two propositions, we include explicit formulas for their 
auto-covariance functions. With the same type of argumentation that will be seen in the proof of Proposition 1, the auto-covariance function can be obtained for other values of $p$.

\begin{proposition}\[\]
If $\left\{ X_{t}\right\} _{t\in \mathbb{R}}\sim FOU\left( \alpha ,\beta
,\sigma ,H\right) $ where $\alpha \neq \beta $, then the auto-covariance
function is

\begin{equation}
\mathbb{E}\left( X_{0}X_{t}\right) =\frac{\sigma ^{2}H}{2}\left[ \frac{%
	\alpha ^{2-2H}f_{H}\left( \alpha t\right) -\beta ^{2-2H}f_{H}\left( \beta
	t\right) }{\alpha ^{2}-\beta ^{2}}\right] . \label{FOUl1l2}
\end{equation}
If $\left\{ X_{t}\right\} _{t\in \mathbb{R}}\sim FOU\left( \alpha ^{\left(
	2\right) },\sigma ,H\right) $, then the
auto-covariance function is

\begin{equation}
\mathbb{E}\left( X_{0}X_{t}\right) =\frac{\sigma ^{2}H}{2\alpha ^{2H}}\left[
\left( 1-H\right) f_{H}\left( \alpha t\right) +\frac{\alpha tf_{H}^{\prime
	}\left( \alpha t\right) }{2}\right] . \label{FOUl1l1}
\end{equation}

\label{covariancesfou2}
\end{proposition}

\begin{proposition}\[\]
If $\left\{ X_{t}\right\} _{t\in \mathbb{R}}\sim FOU\left( \alpha ,\beta
,\gamma ,\sigma ,H\right) $ where $\alpha \neq \beta ,$ $\alpha \neq \gamma ,
$ $\beta \neq \gamma $, then the auto-covariance function is
$$\mathbb{E}\left( X_{0}X_{t}\right) =$$
\begin{equation}
\frac{\sigma ^{2}H}{2}\left[ \frac{%
	\alpha ^{4-2H}f_{H}\left( \alpha t\right) }{\left( \alpha ^{2}-\beta
	^{2}\right) \left( \alpha ^{2}-\gamma ^{2}\right) }+\frac{\beta
	^{2-2H}f_{H}\left( \beta t\right) }{\left( \beta ^{2}-\alpha ^{2}\right)
	\left( \beta ^{2}-\gamma ^{2}\right) }+\frac{\gamma ^{4-2H}f_{H}\left(
	\gamma t\right) }{\left( \gamma ^{2}-\alpha ^{2}\right) \left( \gamma
	^{2}-\beta ^{2}\right) }\right] . \label{FOUl1l2l3}
\end{equation}
If $\left\{ X_{t}\right\} _{t\in \mathbb{R}}\sim FOU\left( \alpha ^{\left(
	2\right) },\beta ,\sigma ,H\right) $ where $\alpha \neq \beta $, then the
auto-covariance function is
$$\mathbb{E}\left( X_{0}X_{t}\right) =$$
\begin{equation}
\frac{\sigma ^{2}H}{2}\left[ \frac{\beta
	^{4-2H}f_{H}\left( \beta t\right) -\alpha ^{4-2H}f_{H}\left( \alpha t\right)
	+\left( \left( 2-H\right) \alpha ^{2-2H}f_{H}\left( \alpha t\right) +\frac{%
		\alpha ^{3-2H}tf_{H}^{\prime }\left( \alpha t\right) }{2}\right) \left(
	\alpha ^{2}-\beta ^{2}\right) }{\left( \alpha ^{2}-\beta ^{2}\right) ^{2}}%
\right] . \label{FOUl1l1l2}
\end{equation}
If $\left\{ X_{t}\right\} _{t\in \mathbb{R}}\sim FOU\left( \alpha ^{\left(
	3\right) },\sigma ,H\right) $, then the auto-covariance function is
$$\mathbb{E}\left( X_{0}X_{t}\right) =$$
\begin{equation}
\frac{\sigma ^{2}H}{4\alpha ^{2H}}\left[
\left( 2-H\right) \left( 1-H\right) f_{H}\left( \alpha t\right) +\left(
7-4H\right) \frac{\alpha t}{2}f_{H}^{\prime }\left( \alpha t\right) +\frac{%
	\alpha ^{2}t^{2}}{4}f_{H}^{\prime \prime }\left( \alpha t\right) \right] .
\label{FOUl1l1l1}
\end{equation}\label{covariancesfou3}
\end{proposition}
\begin{remark}
 With considerably more work, formulas for the auto-covariance function of FOU$(p)$ processes 
 for $p \geq 4$ can be found in the same
 way as shown in the  cases in Proposition \ref{covariancesfou2} and Proposition 
 \ref{covariancesfou3}.
\end{remark}

\section{A simulation study}
In this section we present a small simulation including the cases
FOU$(\lambda_1, \lambda_2, \sigma, H)$ for $\lambda_1 \neq \lambda_2$ 
and FOU$(\lambda^{(2)},\sigma, H)$. In both cases we have simulated $n$ equispaced observations of the FOU processes in $[0,T]$  
 for $T=25,50,100$ and $n=1000,5000,10000$. In each case we have replicated 
 the simulation $m=100$ times. In all cases we have used $\sigma=1$ and
 $\lambda=0.8$ in the FOU$(\lambda^{(2)},\sigma, H)$ case, and $\lambda_1=0.3, 
 \lambda_2=0.8$ in the   FOU$(\lambda_1, \lambda_2, \sigma, H)$ case. 
 According with Theorem \ref{consistency_theorem}, in all cases we have used
 $w(x)=\frac{|x|^{2p}}{1+|x|^{2p+3}} $ and the order $2$ Daubechies' filter $a=$\\
 $ \frac{1}{\sqrt(2)}(.482962,-.836516,.224143,.129409)$.
 Even though the FOU$(p)$ process has a short range dependence for $p \geq 2$ and every value of $H$,
 if $H>1/2$ the increments 
 of the fractional
 Brownian motion that drives the FOU$(p)$ process have a long range dependence. For this 
 reason, in order to get an idea as to whether the true value of $H$ 
 influences the accuracy of the parameter estimates, we have considered three values of
 $H:$ $0.3, 0.5$ and $0.7$. 
 \subsection{Consistency of the estimators}
 In Tables \ref{tH03ll} to \ref{tH07ll} we report the mean and the deviation of each estimator 
 for $m=100$ replications in the 
 FOU$(\lambda^{(2)},\sigma, H)$ for $H=0.3$, $H=0.5$ and $H=0.7$ respectively. 
 Similarly, Tables \ref{tH03l1l2} to \ref{tH07l1l2}   refer to the case of
 FOU$(\lambda_1, \lambda_2, \sigma, H)$.\\ 
 Table \ref{tH03ll} shows that $H$ and $\sigma$ are well estimated for all values of $T$ and $n$ 
 considered. Concerning $\lambda$, we observe that it is necessary to take large values of $T$ 
 and $n$ in order to obtain good estimates. We observe that the  relative deviation of $\widehat{\lambda}$ is 
 greater  than the deviations  of $\widehat{\sigma}$ and $\widehat{H}$. This is reasonable 
 because $\lambda$  is estimated as a function of  $\widehat{\sigma}$ and $\widehat{H}$, and so
 the 
 error of the estimation will be greater. Also, Table \ref{tH03ll} shows that the deviations of 
 $\widehat{\sigma}$ and $\widehat{H}$ decrease as $T$ and $n$ increase. The same
 is true for $\hat{\lambda}$ for $T=100$ and $T=50$. But when $T=25$, it does not seem that
 the deviations are decreasing with $n$ and the estimation is not very good.
 The same remarks are valid 
 for Table \ref{tH05ll} and Table \ref{tH07ll}. Therefore, these results show that there
 are no 
 substantial differences in the efficiency for the estimator of the parameters for values 
 greater
 or smaller than $H=0.5$. That is, the efficiency of the estimators does not depend on of the 
 irregularity
 of the trajectories of the fractional Brownian motion which drives the FOU processes.
 Columns $3$ and $4$ of Tables \ref{tH05l1l2} to \ref{tH07l1l2} are very similar to
 the same columns in Tables \ref{tH03ll} to \ref{tH07ll}. This is reasonable, because 
 $\sigma$ and $H$
 were estimated independently of the FOU$(p)$ model to adjust.
 Concerning the estimators of
 $\lambda_1$ and $\lambda_2$ we observe that the speed of convergence is slower 
 than for the case
 where there is  only one $\lambda$ to estimate. Also, the relative deviations
 of $\widehat{\lambda}_1$
 and $\widehat{\lambda}_2$ are greater than in  the previous case. This is expected to happen because
 it is well known that the more parameters 
 a model has, the more deviation its estimators will have.

\begin{table*}[ht]\caption{ Mean estimation (with corresponding deviations) of the parameter 
for a FOU$(\lambda^{(2)},H, \sigma)$ viewed at $n$ equispaced points of $[0,T]$, 
where $\lambda=0.8$, $H=0.3$ and $\sigma =1$  for $m=100$  replications. }
 \label{tH03ll}

\centering
\begin{tabular}{|c|c||c|c|c|}

    \hline
   $T$ & $n$ & $\hat{H}$   & $\hat{\sigma}$  & $\hat{\lambda}$    \\
   \hline 
   100 & 1000 &  0.2974 (0.037) &  0.9197 (0.082) &  0.7536 (0.218)  \\
           & 5000 & 0.3004 (0.017) &0.9877  (0.066) & 0.7955 (0.167)\\
           & 10000 & 0.3008 (0.013) & 0.9961 (0.058) & 0.8265 (0.146) \\
           \hline
     50 & 1000 & 0.2983 (0.037) & 0.9618 (0.111)  & 0.7713 (0.270)   \\
            & 5000 & 0.3007  (0.017) &0.9984  (0.078)  &0.8249  (0.238)    \\
            & 10000 & 0.3005 (0.012) & 0.9997 (0.065) & 0.8199 (0.198)\\

   \hline
     25 & 1000 & 0.2904 (0.032) & 0.9518 (0.068)  & 0.7205 (0.255) \\
            & 5000  &0.3007  (0.017) & 1.0037  (0.091)  &0.8605  (0.284)  \\
            & 10000 & 0.9997 (0.012) & 1.0006 (0.072) & 0.8742 (0.261) \\
   \hline

\end{tabular}
\end{table*}

\begin{table*}[ht]\caption{ Mean estimation (with corresponding deviations) of the parameter 
for a FOU$(\lambda^{(2)},H, \sigma)$ viewed at $n$ equispaced points of $[0,T]$, 
where $\lambda=0.8$, $H=0.5$ and $\sigma =1$  for $m=100$  replications. }
 \label{tH05ll}

\centering
\begin{tabular}{|c|c||c|c|c|}

    \hline
   $T$ & $n$ & $\hat{H}$   & $\hat{\sigma}$  & $\hat{\lambda}$   \\
   \hline 
   100 & 1000 &  0.4894 (0.035) & 0.9901 (0.082) &  0.7514 (0.197)  \\
           &  5000 & 0.4993 (0.016) & 0.9829 (0.065) & 0.7969 (0.184) \\
           & 10000 & 0.4993 (0.011) & 0.9938 (0.057) & 0.8159 (0.162) \\
           
           \hline
     50 & 1000 & 0.4924 (0.034) & 0.9396 (0.107)  & 0.7673 (0.263)   \\
            & 5000 & 0.5002  (0.014) & 0.9965 (0.072)  & 0.8358 (0.213)    \\
            & 10000 & 0.5005 (0.012) & 1.0024 (0.068) & 0.8135 (0.197) \\
   \hline
     25 & 1000 &  0.4860 (0.035) & 0.8883 (0.085)  & 0.7331 (0.215) \\
            & 5000 & 0.4998 (0.016)& 0.9985 (0.088)  & 0.8541  (0.231)  \\
            & 10000 & 0.4989 (0.010) & 0.9936 (0.064) & 0.8153 (0.285) \\
   \hline

\end{tabular}
\end{table*}

\begin{table*}[ht]\caption{ Mean estimation (with corresponding deviations) of the parameter 
for a FOU$(\lambda^{(2)},H, \sigma)$ viewed at $n$ equispaced points of $[0,T]$, 
where $\lambda=0.8$, $H=0.7$ and $\sigma =1$ for $m=100$  replications. }
 \label{tH07ll}

\centering
\begin{tabular}{|c|c||c|c|c|}

    \hline
   $T$ & $n$ & $\hat{H}$   & $\hat{\sigma}$  & $\hat{\lambda}$   \\
   \hline 
   100 & 1000 & 0.6818 (0.036) & 0.8865 (0.107) & 0.7587 (0.196)  \\
           &  5000 &0.7001  (0.015) & 0.9875  (0.074) & 0.7985  (0.121) \\
           & 10000 & 0.7013 (0.009) & 0.9996 (0.053) & 0.8708 (0.121) \\
            \hline
     50 & 1000 & 0.6953 (0.036) & 0.955 (0.147)  & 0.7902 (0.202)   \\
            & 5000 & 0.6995 (0.015) & 0.9933  (0.087)  & 0.8379  (0.187)    \\
            & 10000 & 0.6991 (0.011) & 0.9931 (0.068) & 0.7929 (0.199) \\

   \hline
     25 & 1000 & 0.6878 (0.033) & 0.8877 (0.097)  & 0.8322 (0.344) \\
            & 5000  &0.7008  (0.015) & 1.0065 (0.095)  & 0.8490 (0.243)  \\
            & 10000 & 0.6990 (0.010) & 0.9922 (0.069) & 0.8410 (0.232) \\

   \hline

\end{tabular}
\end{table*}

\begin{table*}[ht]\caption{ Mean estimation (with corresponding deviations) of the parameter 
for a FOU$(\lambda_1, \lambda_2,H, \sigma)$ viewed at $n$ equispaced points of $[0,T]$, 
where $\lambda_1=0.3, \lambda_2=0.8$, $H=0.3$ and $\sigma =1$ for $m=100$  replications. }
 \label{tH03l1l2}

\centering
\begin{tabular}{|c|c||c|c|c|c|}

    \hline
   $T$ &  $n$ & $\hat{H}$   & $\hat{\sigma}$ & $\hat{\lambda_1}$  &  $\hat{\lambda_2}$   \\
   \hline 
   100 & 1000 & 0.2949 (0.036)  & 0.9413 (0.080)  & 0.2285 (0.279)  & 0.7023 (0.439)    \\
           &  5000 & 0.2990 (0.016) & 0.9880 (0.060)   & 0.2769 (0.272)  & 0.7694 (0.349)  \\
           & 10000 &0.3019 (0.011) & 1.0040 (0.050) & 0.3245 (0.295)  & 0.7326 (0.356)  \\

           \hline
     50 & 1000 & 0.2970 (0.034) & 0.9678 (0.106) & 0.2255 (0.267)  & 0.7431 (0.522)  \\
            & 5000 & 0.2992 (0.015) & 0.9944 (0.071)  & 0.2728 (0.296)    & 0.8038 (0.436)    \\
            &10000 &0.3009 (0.012) & 1.0480 (0.067) & 0.3078 (0.333) & 0.7577 (0.453) \\

   \hline
     25 & 1000 & 0.3069 (0.037)  & 1.0217 (0.142) & 0.2096 (0.301)  & 0.9300 (0.708)   \\
            & 5000 & 0.2991 (0.015) & 0.9973 (0.082)  &  0.2797 (0.327)  & 0.8511 (0.552)   \\
            &10000 & 0.2971 (0.010) & 0.9843 (0.043) & 0.2379 (0.303)  & 0.8518 (0.617)  \\
   \hline

\end{tabular}
\end{table*}

\begin{table*}[ht]\caption{ Mean estimation (with corresponding deviations) of the parameter 
for a FOU$(\lambda_1, \lambda_2,H, \sigma)$ viewed at $n$ equispaced points of $[0,T]$, 
where $\lambda_1=0.3, \lambda_2=0.8$, $H=0.5$ and $\sigma =1$ for $m=100$  replications. }
 \label{tH05l1l2}

\centering
\begin{tabular}{|c|c||c|c|c|c|}

    \hline
   $T$ & $n$ & $\hat{H}$   & $\hat{\sigma}$ & $\hat{\lambda_1}$  &  $\hat{\lambda_2}$   \\
   \hline 
   100 & 1000 & 0.4905 (0.034)  & 0.9234 (0.086)  & 0.2183 (0.245)  & 0.7041 (0.478)  \\
           &  5000 & 0.5027 (0.014) & 1.0011 (0.065)   & 0.2451 (0.249)  & 0.8152 (0.371)   \\
           & 10000 & 0.5012 (0.011) & 1.0028 (0.054) & 0.2815 (0.273)  & 0.8033 (0.333)  \\
           
           \hline
     50 & 1000 & 0.4999 (0.031) & 0.9751 (0.099) & 0.2949 (0.275)   & 0.7161 (0.446)   \\
            & 5000 & 0.4995 (0.013) & 0.9940 (0.065)  & 0.2906 (0.279)   & 0.8284 (0.458)   \\
            & 10000 & 0.5013 (0.008) & 1.0071 (0.062) & 0.2533 (0.278) & 0.7845 (0.422)   \\
   \hline
     25 & 1000 & 0.5054 (0.037)  & 1.1071 (0.154) & 0.3047 (0.364) & 0.9028 (0.629)   \\
            & 5000 & 0.5008 (0.014) & 1.0048 (0.084)  & 0.2558 (0.335)  & 0.9129 (0.560)   \\
            & 10000 & 0.5014 (0.011) & 1.0102 (0.069) & 0.2659 (0.296) & 0.7811 (0.530)  \\
   \hline

\end{tabular}
\end{table*}

\begin{table*}[ht]\caption{ Mean estimation (with corresponding deviations) of the parameter
for a FOU$(\lambda_1, \lambda_2,H, \sigma)$ viewed at $n$ equispaced points of $[0,T]$, 
where $\lambda_1=0.3, \lambda_2=0.8$, $H=0.7$ and $\sigma =1$ for $m=100$  replications. }
 \label{tH07l1l2}

\centering
\begin{tabular}{|c|c||c|c|c|c|}

    \hline
   $T$ & $n$  & $\hat{H}$   & $\hat{\sigma}$  & $\hat{\lambda_1}$  &  $\hat{\lambda_2}$    \\
   \hline 
  100 & 1000 & 0.6918 (0.034) & 0.9254 (0.105) & 0.2747 (0.271)   & 0.6787 (0.371)    \\
           &  5000 &0.7004 (0.015)  & 0.9936 (0.073)   & 0.3025 (0.269)  & 0.7434 (0.305)   \\
           & 10000 & 0.7004 (0.010) & 0.9985 (0.059) & 0.3074 (0.247)  & 0.7768 (0.309) \\
           
           \hline
     50 & 1000 & 0.6988 (0.033) & 0.9771 (0.133) & 0.3086 (0.296)  & 0.7825 (0.426)   \\
            & 5000 & 0.7006 (0.015) & 1.0017 (0.084)  & 0.3125 (0.291)   & 0.8051 (0.411)    \\
            & 10000 & 0.7002 (0.011) & 1.0015 (0.068) & 0.3181 (0.292) & 0.7577 (0.386) \\
   \hline
     25 & 1000 & 0.7008 (0.033)  & 1.0032 (0.159) & 0.2830 (0.347) & 0.9254 (0.703) \\
            & 5000 & 0.7007 (0.015) & 1.0065 (0.095)  &  0.2667 (0.312)  & 0.8972 (0.558)    \\
            & 10000 & 0.7002 (0.011) & 1.0038 (0.076) & 0.3009 (0.361) & 0.7685 (0.522) \\
   \hline

\end{tabular}
\end{table*}

\subsection{Asymptotic distribution of the estimators}
About the asymptotic distribution of the estimators, we have that $\hat{H}$ and $\hat{\sigma}$
have asymptotic Gaussian distributions (Theorem \ref{asymptotic of H sigma}) 
   and this was corroborated by the simulations. The Truncated Cram\'er von-Mises test of 
   normality (\cite{tcvm}) does not reject normality for any of the cases, including those in
   Tables 1 to 6. Concerning the asymptotic distribution for the estimator for 
    the $\lambda$, we have observed that normality is not rejected  when
   we have only one parameter $\lambda$  to estimate, but when there are two or more parameters
   to estimate, normality is rejected. In Table \ref{pvalue_lambda} we present the p-values for  
   the truncated Cram\'er-von Mises test of normality for  $\hat{\lambda}$ in the FOU$(\lambda^{(2)},H, \sigma)$
   case viewed 
   at $n$ equispaced points of $[0,T]$, 
where $ \lambda=0.8$, $\sigma=1$ and $H=0.7$  for $m=100$  replications.
For other values of the parameters, the results are similar.
In Figure \ref{densities_lambda} we presented the estimation of the density 
 for the cases given in Table \ref{pvalue_lambda}.
From Table \ref{pvalue_lambda} and Figure \ref{densities_lambda} we observe that 
the simulations confirm that the hypothesis that 
it is necessary to consider large
$T$  and small $T/n$, for example 
when $T=25$ the convergence of $\hat{\lambda}$ to $\lambda$  is not clear, and simillarly when  
$T=100$ and $n=1000$ (in this case $T/n=0.1$ is not small enough). In Table \ref{pvalue_lambdas} we observe that
normality is clearly rejected in all the cases considered even for large values of $n$ and $T$.
In Figure \ref{densities_lambdas}, we observe the estimated densities for 
$\hat{\lambda_1}$ and $\hat{\lambda_2}$.
This could be happen because $\hat{\lambda}$ is asymptotically Gaussian if the process is
observed on the entire the interval
$[0,T]$  when $T \rightarrow +\infty$ (Theorem \ref{asymptotic lambdas}). When we estimate
$\lambda$ by discretization, there 
is introduced a remainder $\hat{\lambda}_{T_n}-\hat{\lambda}_{T_n}^{(n)}$ that can introduce a 
bias in the asymptotic 
distribution. On the other hand, Tables 4, 5 and 6 suggest that the consistency is more difficult to detect in the 
$\lambda_1 \neq \lambda_2$ 
case compared to the case where there is only one $\lambda$ to estimate. This slow consistency
may explain the lack of goodness of fit to the Gaussian distribution. 
It is reasonable to expect the asymptotic normality of $\hat{\lambda}$ at the cost to adding
some relation between $T$ and $n$, but that seems 
to be difficult to detect in practice in light of the simulations performed. 
In Figure \ref{densities_lambdas}, we observe the estimation of the densities of
$\widehat{\lambda}_1$ on the left
for $T=25,50,100$ where $n=1000$ in black and $n=5000$ in blue, similarly for $\hat{\lambda}_2$ in the three 
graphs on the left.

   \begin{table*}[ht]\caption{ p-values for the Truncated Cram\'er-von Mises test of normality
    for $\hat{\lambda}$ for the  FOU$(\lambda^{(2)},H, \sigma)$ model viewed at $n$ equispaced points of $[0,T]$, 
where $ \lambda=0.8$, $H=0.7$ and $\sigma =1$ for $m=100$  replications. }
 \label{pvalue_lambda}

 \centering
\begin{tabular}{|c|c||c|}

    \hline
   $T$ & $n$ & p-value for $\hat{\lambda}$    \\
   \hline 
  $25$ & $1000$  &  0.006    \\
           &  $5000$ &  0.143   \\
           & $10000$ & 0.005     \\
           
           \hline
       $50$ & $1000$  &   0.254     \\
           &  $5000$ &  0.678   \\
           & $10000$ &  0.103    \\
           
           \hline
  
     $100$ & $1000$  &  0.826     \\
           &  $5000$ &  0.790    \\
           & $10000$ & 0.854     \\
           
           \hline
     
\end{tabular}
\end{table*}

\begin{figure}[H]
 \centering
   \includegraphics[scale=0.41]{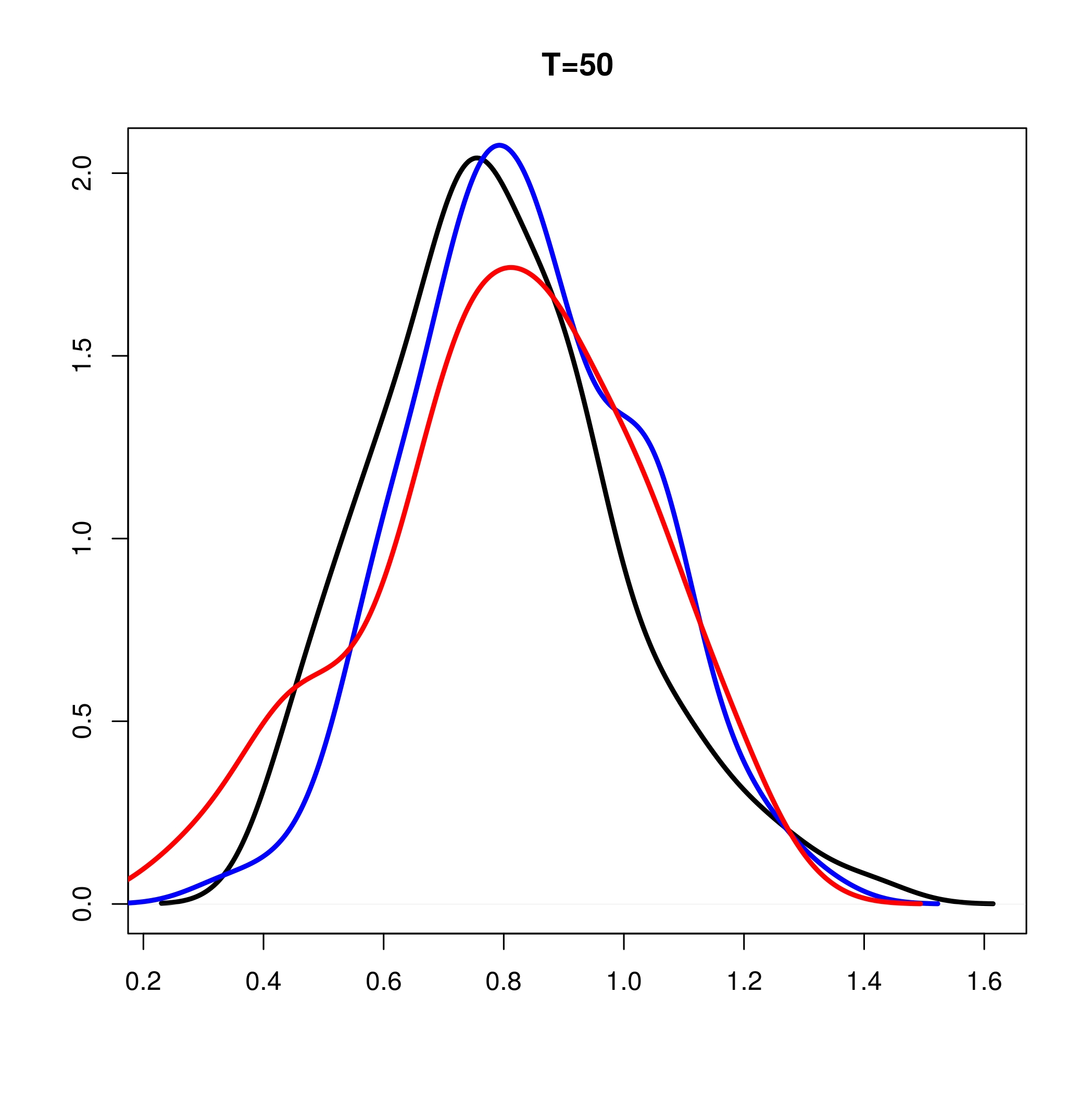} 
   \includegraphics[scale=0.41]{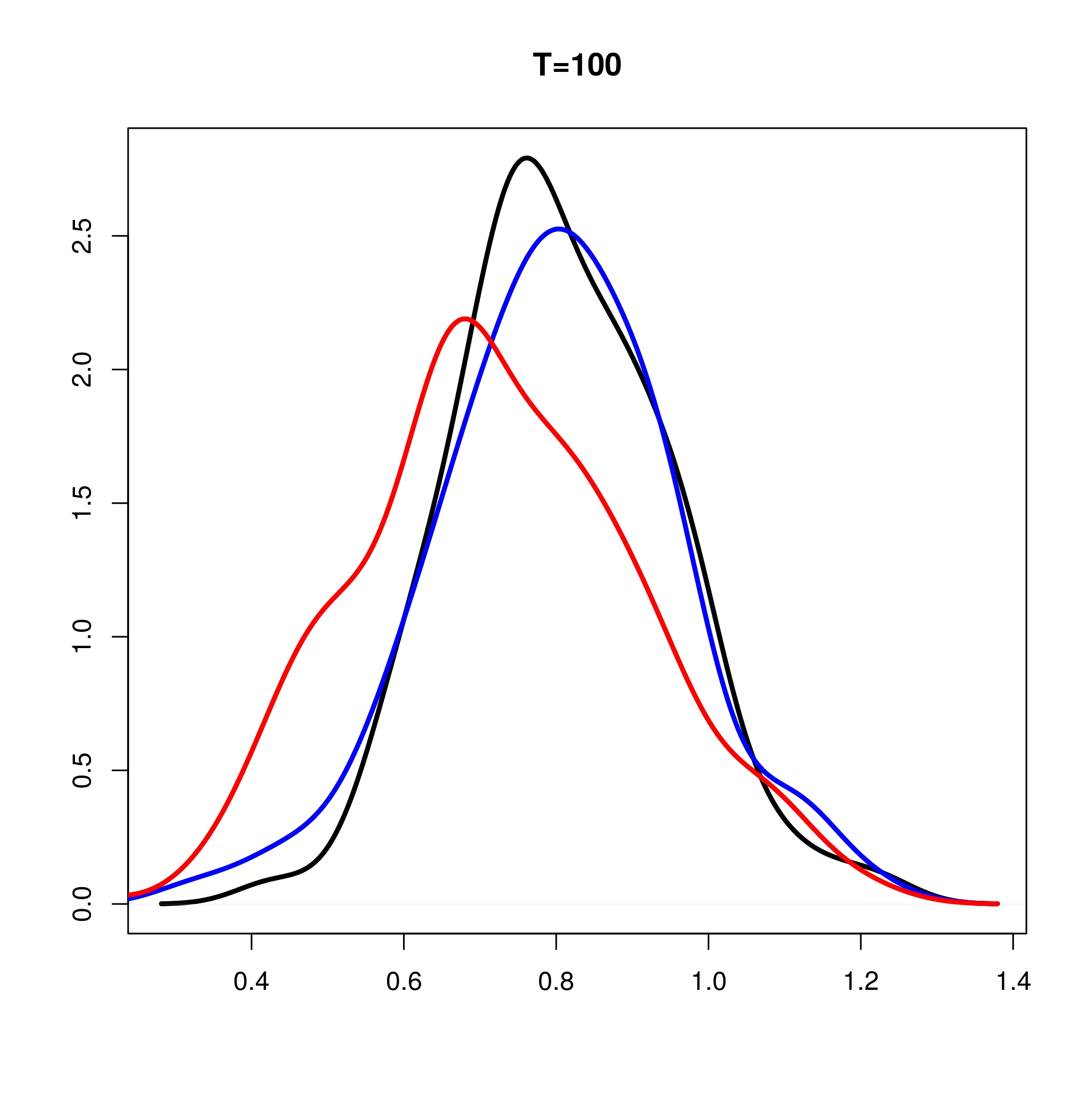} 
   \caption{Estimated densities for $\hat{\lambda}$ for different values of $T$  for the   
    FOU$(\lambda^{(2)},\sigma,H)$ for $(\lambda,\sigma,H)=(0.8,1,0.7)$ where $n=1000$ (red), $n=5000$ (blue), 
    $n=10000$ (black). 
    } 
  \label{densities_lambda}
\end{figure}

\begin{figure}[H]
 \centering
   \includegraphics[scale=0.7]{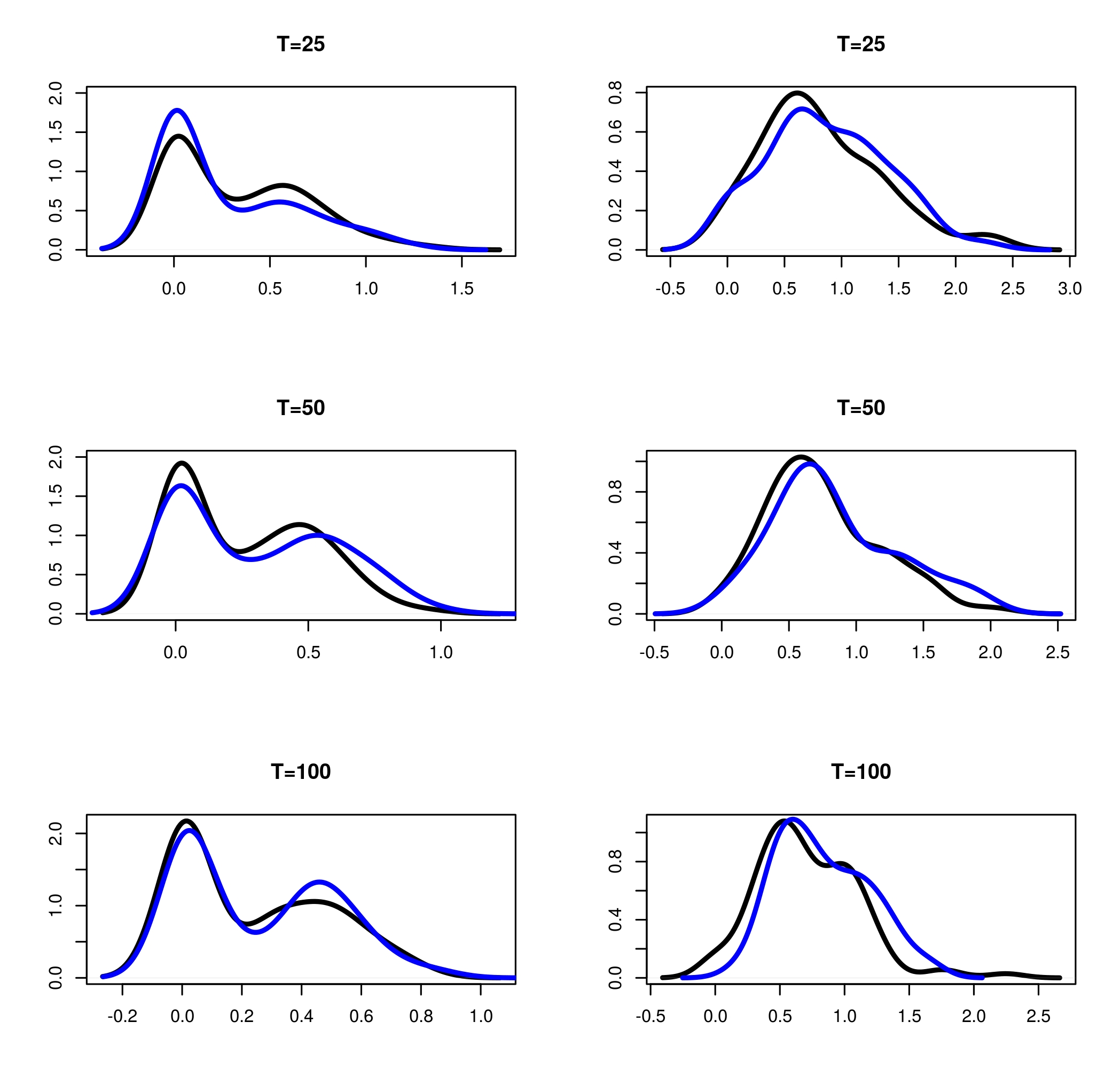} 
   \caption{Estimated densities for $\hat{\lambda}_1$ (left) and $\hat{\lambda}_2$ (right) for different values of $T$  for the   
    FOU$(\lambda_1,\lambda_2, \sigma,H)$ for $(\lambda,\sigma,H)=(0.3, 0.8,1,0.7)$ where $n=1000$ (black) and $n=5000$ 
    (blue). 
    } 
  \label{densities_lambdas}
\end{figure}

\begin{table*}\caption{ p-values for the Truncated Cram\'er-von Mises test of normality
    for $\hat{\lambda_1}$ and $\hat{\lambda_2}$ for the FOU$(\lambda_1, \lambda_2,H, \sigma)$ 
    model viewed at $n$ equispacied points of $[0,T]$, 
where $ \lambda_1=0.3, \lambda_2=0.8$, $H=0.7$ and $\sigma =1$ for $m=100$  replications. }
\label{pvalue_lambdas}
\centering
\begin{tabular}{|c|c||c|c|}

    \hline
   $T$ & $n$ \ \ \ & p-value for $\hat{\lambda_1}$  & p-value for $\hat{\lambda_2}$   \\
   \hline 
  $25$ & $1000$  & 0.000   &  0.001    \\
           &  $5000$ & 0.000   & 0.001   \\
           & $10000$ & 0.000   & 0.002    \\
           
           \hline
       $50$ & $1000$  & 0.000    & 0.001    \\
           &  $5000$ & 0.000  & 0.002   \\
           & $10000$ & 0.000   & 0.001    \\
           
           \hline
  
     $100$ & $1000$  & 0.000    & 0.001   \\
           &  $5000$ & 0.000   & 0.002    \\
           & $10000$ & 0.000   & 0.002    \\
           
           \hline
     
\end{tabular}
\end{table*}

\section{Modelling an observed time series using FOU$(p)$ processes}
Given $X_1, X_2,...,X_n$, observations of an stationary centered time series that we want to model using FOU$(p)$ process, firstly we need to consider the observations as an equispaced sample on the interval $[0,T]$, that is $X_{T/n},X_{2T/n},...,X_{T}$ for some value of $T$.
According to what was seen in the previous section,  we need to estimate the parameters $\sigma$, $H$ and $\lambda$ whose estimators depends
on $T$ (except $H$). Thus, firstly we need to know
the value of $T$. 
\subsection{Choosing the value of \textbf{$T$}}
Give the value of $T$ is give the unit of measurement in which the observations 
are taking. Although in every cases it is natural to take a certain value of $T$ (for example,
if the observations are monthly and we have $120$ observations, it is natural to take $T=120$ months or $T=10$ years) we can easily take any value of $T$ and interpret it in terms 
of the original time measure of the data. Therefore,  we can take
advantage of this fact, choosing a value of $T$ for which the goodness of fit of the model is the 
best possible according to certain criteria.
As we have seen in the previous section, to model a time series data set from a FOU$(p)$ 
processes it is necessary to have  values of $n$ and $T$ sufficiently large so that $T/n$ is 
small. Now, $n$ is
the sample size and we assume that the observations lie in some interval $[0,T]$. Although Theorem \ref{asymptotic lambdas} suggests
that $T=n^{1-\alpha}$ for a certain value of $\alpha$, the asymptotic result remains valid for 
$T=cn^{1-\alpha}$ for any value of a constant $c>0$. The real data set contains a 
fixed value of $n$, so it is better in each particular case to 
optimize some criterion to obtain a suitable value of $T$. For example, it is convenient to 
choose
a value of $T$ that minimizes an MAE or RMSE, or a value of $T$ that maximizes the Willmott index.
In the following section, we will apply these criteria to three data sets.

\section{Applications to real data}

In this section we analise three real data sets. In each of them, we
adjusted different FOU$\left( p\right) $ models \ for $p=2,3,4$, and ARMA
models. To fit the FOU$\left( p\right) $ model, we suppose that the real
data set, is indexed in the interval $\left[ 0,T\right] $ for a suitable value of $T.$ We
also asume in all of cases that the observations are
equally spaced in time, that is:\ $X_{T/n},X_{2T/n},...,X_{T}.$ 
To estimate the parameters of each FOU$\left( p\right) $, we  apply the procedure suggested in 
the previous section.
In each case, we also fit different ARMA (or ARFIMA) models, and we compare the
performance of these ARMA (or ARFIMA) models with that of the FOU models, through four measures
of the quality of prediction: the root mean square error of prediction for the last $m$ observations,
that is \[RMSE=\sqrt{\frac{1}{m}\sum_{i=1}^{m}\left( X_{n-m+i}-\widehat{X}%
_{n-m+i}\right) ^{2}};\] the mean absolute error of prediction for last $m$
observations and their respective predictions, that is, \[MAE=\frac{1}{m}%
\sum_{i=1}^{m}\left\vert X_{n-m+i}-\widehat{X}_{n-m+i}\right\vert  \] the
Willmott index  (\cite{Wilmott}) defined by 
\[W_2=1-\frac{\sum_{i=1}^{m}%
\left( X_{n-m+i}-\widehat{X}_{n-m+i}\right) ^{2}}{\sum_{i=1}^{m}\left(
\left\vert \widehat{X}_{n-m+i}-\overline{X}(m)\right\vert +\left\vert 
{X}_{n-m+i}-\overline{X}(m)\right\vert \right) ^{2}}\] and the Wilmott $L^{1}$ index, defined by
\[ W_{1}=1-\frac{\sum_{i=1}^{m}\left\vert
X_{n-m+i}-\widehat{X}_{n-m+i}\right\vert }{\sum_{i=1}^{m}\left( \left\vert 
\widehat{X}_{n-m+i}-\overline{X}(m)\right\vert +\left\vert {X}%
_{n-m+i}-\overline{X}(m)\right\vert \right) };\]
 where $\overline{X}(m):=%
\frac{1}{m}\sum_{i=1}^{m}X_{n-m+i},$ and $X_{1},X_{2},...,X_{n}$ $\left( 
\text{or }X_{T/n},X_{2T/n},...,X_{T}\right) $ are the real observations, while $%
\widehat{X}_{i}$ are the predictions given by the model for the value $%
X_{i}. $
All the predictions considered are one step.
For these three cases, we will compare the graphs of the empirical auto-covariance
function with
those of some fitted models.

Firstly, in the following subsection we suggest how to choose a suitable value of $T$
to model $n$ observations using an FOU$(p)$ model.

\subsection{Box, Jenkins and Reinsel ``Series A"}

The Series A is a record of $197$ chemical process concentration readings,
taken every two hours. This series was introduced by \cite{Box}, who suggest using an 
ARMA$\left( 1,1\right) $ process to model this data set. An
AR$\left( 7\right) $ is proposed in \cite{Cleveland} and \cite{McLeod}. 
In Figure \ref{acfseriesA} we observe that the auto-covariance function of the AR$(7)$ and 
ARMA$(1,1)$ adjusted models goes to zero very quickly
and their auto-covariance structure does not resemble that observed. 
To obtain a suitable value of $T$, we calculate the RMSE, MAE and the two indices of Wilmott for 
values
of $T$ between $5$ and $25$. In each case, we estimate the parameters and calculate the 
four measures of  the quality of prediction for $m=50$ predictions. In Figure \ref{fou_lambda2}
we show the values of the four measures for values of $T$ between $7$ and $25$ 
when we adjusted an
FOU$(\lambda^{(2)},\sigma,H)$ model (the values of $T=5$ and $T=6$ had very bad
performance and are not included in the figure). Observe that in the four cases, the 
optimal value is reached for $T=11$. Also, we can use a neigbourhood of $T=11$ and
we have similar performance.
 In the rest of the adjusted FOU cases, the optimal value
was reached in a neighbourhood at $T=12$ or $T=7$  depending on which measure was optimized.
In Table \ref{tablaseriesA} we show the values of $W_2$, $RMSE$, $W_{1}$  and $MAE$ for AR$(7)$,
ARMA$(1,1)$ and different FOU$(p)$ for
$p=2,3,4$. In all the FOU processes considered, we use $T=12$. For the
estimation of $\lambda$, we have
used the $constrOptim$ function of the $R$ package with the conditions $0.01 \leq \lambda_i \leq 1.5$ 
(to optimize
on a compact $\Lambda$) and 
$\lambda_{i+1} \geq \lambda_i+0.01$ (to ensure that $\lambda_i < \lambda_{i+1}$) for $i=1,2,3$.
The first results of the estimation are
$\hat{H}=0.1367, \hat{\sigma}=0.5464.$ For the $\lambda$, one estimates 
$\hat{\lambda}=0.1554$ in
FOU$(\lambda^{(2)},\sigma,H)$, $0.1250$ in
FOU$(\lambda^{(3)},\sigma,H)$, $0.1076$  in
FOU$(\lambda^{(4)},\sigma,H)$, $(0.0328,0.2273)$ in
FOU$(\lambda_1,\lambda_2,\sigma,H)$, $(0.0123,0.0241,0.2291)$ in
FOU$(\lambda_1,\lambda_2,\lambda_3,\sigma,H)$ and $(0.0267,0.0533,0.1163,0.1766)$
in
FOU$(\lambda_1,\lambda_2,\lambda_3,\lambda_4,\sigma,H)$.

\begin{table}\caption{Values of $W_2$, $RMSE$, $W_{1}$  and $MAE$ for different models adjusted to Series A.}
 \label{tablaseriesA}
 \centering
$%
\begin{array}{|c|cccc|}
\hline
\text{Model} & W_2 & RMSE & W_{1} & MAE \\ 
\hline
\text{AR}\left( 7\right) & 0.6184  & \textbf{0.2995} & \textbf{0.4943} & \textbf{0.2167} \\ 
\text{ARMA}\left( 1,1\right) & 0.5883 & 0.3120 & 0.4620 & 0.2343 \\ 
\text{FOU}\left( \lambda _{1},\lambda _{2},\sigma ,H\right) & 0.6263 &0.3076 
&0.4743 & 0.2372 \\ 
\text{FOU}\left( \lambda _{1},\lambda _{2},\lambda _{3},\sigma ,H\right) & 0.6260
 &0.3076 &0.4743  &0.2371  \\ 
\text{FOU}\left( \lambda _{1},\lambda _{2},\lambda _{3},\lambda _{4},\sigma
,H\right) & 0.6244 & 0.3074  & 0.4733 &0.2369  \\ 
\text{FOU}\left( \lambda ^{\left( 2\right) },\sigma ,H\right) & 0.6247  &0.3086 
 &0.4712  &0.2393  \\ 
\text{FOU}\left( \lambda ^{\left( 3\right) },\sigma ,H\right) & \textbf{0.6277} & 0.3078
 & 0.4750  &0.2373 \\ 
\text{FOU}\left( \lambda ^{\left( 4\right) },\sigma ,H\right) &0.6264 &0.3076
&0.4742  &0.2372 \\
\hline
\end{array}%
$ 
  \end{table}
  
We observe that  FOU$(\lambda^{(3)},\sigma,H)$ has the best performance in 
terms of the $L^{2}$-Willmott Index,
and its RMSE is very close to that of the AR$(7)$ model. For the $L^{1}$ the performance
is slightly worse than the AR$(7)$ model.

On the other hand, in Figure \ref{acfseriesA} we observe that the auto-covariances of the 
AR$(7)$ and ARMA$(1,1)$ adjusted models, go to zero very quickly
and their auto-covariance structure does not resemble that observed. Besides the adjusted
FOU$(\lambda^{(3)},\sigma,H)$ and 
FOU$(\lambda^{(4)},\sigma,H)$ have a better performance. 

\begin{figure}[]
 \centering
   \includegraphics[scale=0.5]{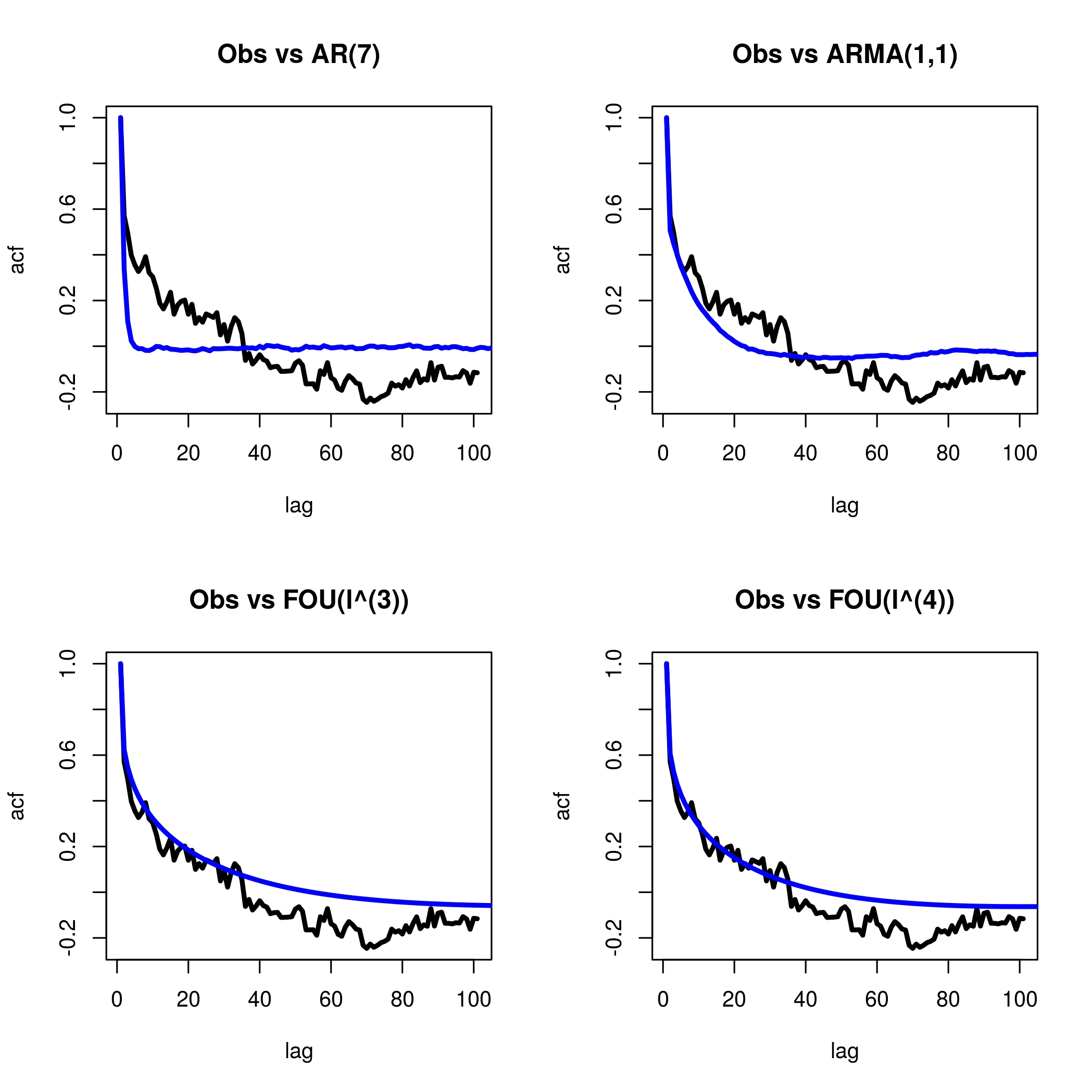} 
   \caption{Empirical auto-covariance function (black) vs fitted auto-covariance function (blue) according to the adjusted model for 
   series A data set.} 
  \label{acfseriesA}
\end{figure}
\begin{figure}[]
 \centering
   \includegraphics[scale=0.6]{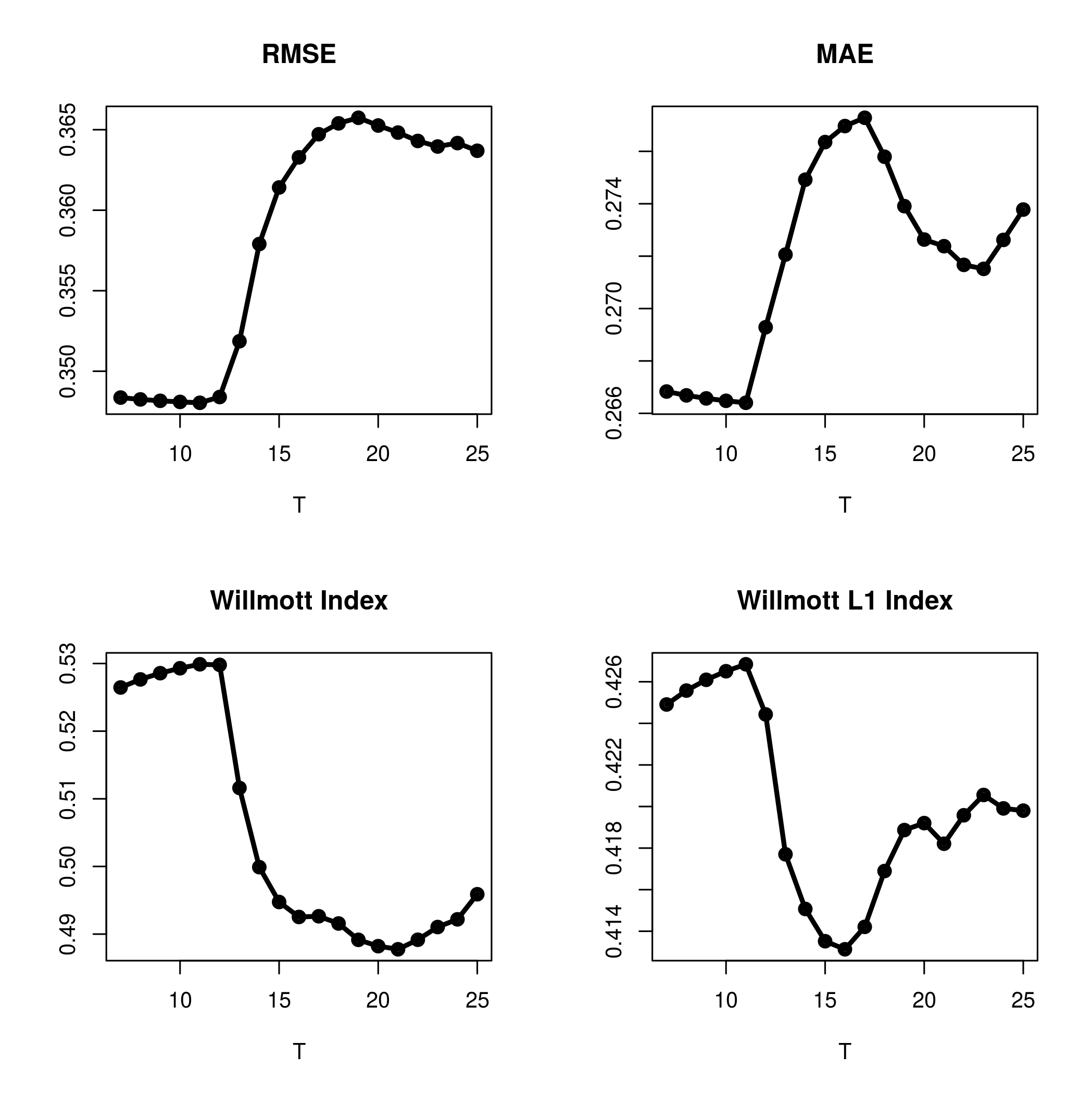} 
   \caption{RMSE, MAE, and the two indices of Willmott for $m=50$ predictions when the model used
   is FOU$(\lambda^{(2)},\sigma, H)$ for different values of $T$.} 
  \label{fou_lambda2}
\end{figure}

\subsection{Water level of Lake Huron}

The water level in feet of Lake Huron for 1875--1972, is a time series of $98$ observations. 
This is a small size
to apply our procedure of estimation, which requires $n,T\rightarrow +\infty$ and $T/n\rightarrow 0$.
Then, we can apply Corollary \ref{estimo_solo_lambda}, for different values of $H$ and $\sigma=1$ 
(assuming that the fractional Brownian motion which drives the FOU process is standard). The results
for $H=0.5, 0.6, 0.7$ and $T=10, 20, 30$ were similiar. 
Using the procedure to choose the value of $T$ proposed in subsection 5.1, we have obtained that
the best performance was for $T=30$.

The series has a slight trend, which  was removed before
adjusting the models. In \cite{Brockwell}, it is suggested to use an AR$\left(
2\right) $ and ARMA$\left( 1,1\right) $ for this series. 
Nor are there significant differences between the observed curve and the predictions curve for the the different models in the last 20 
observations (Figure \ref{Huron40}).

\begin{figure}[]
 \centering
   \includegraphics[scale=0.49]{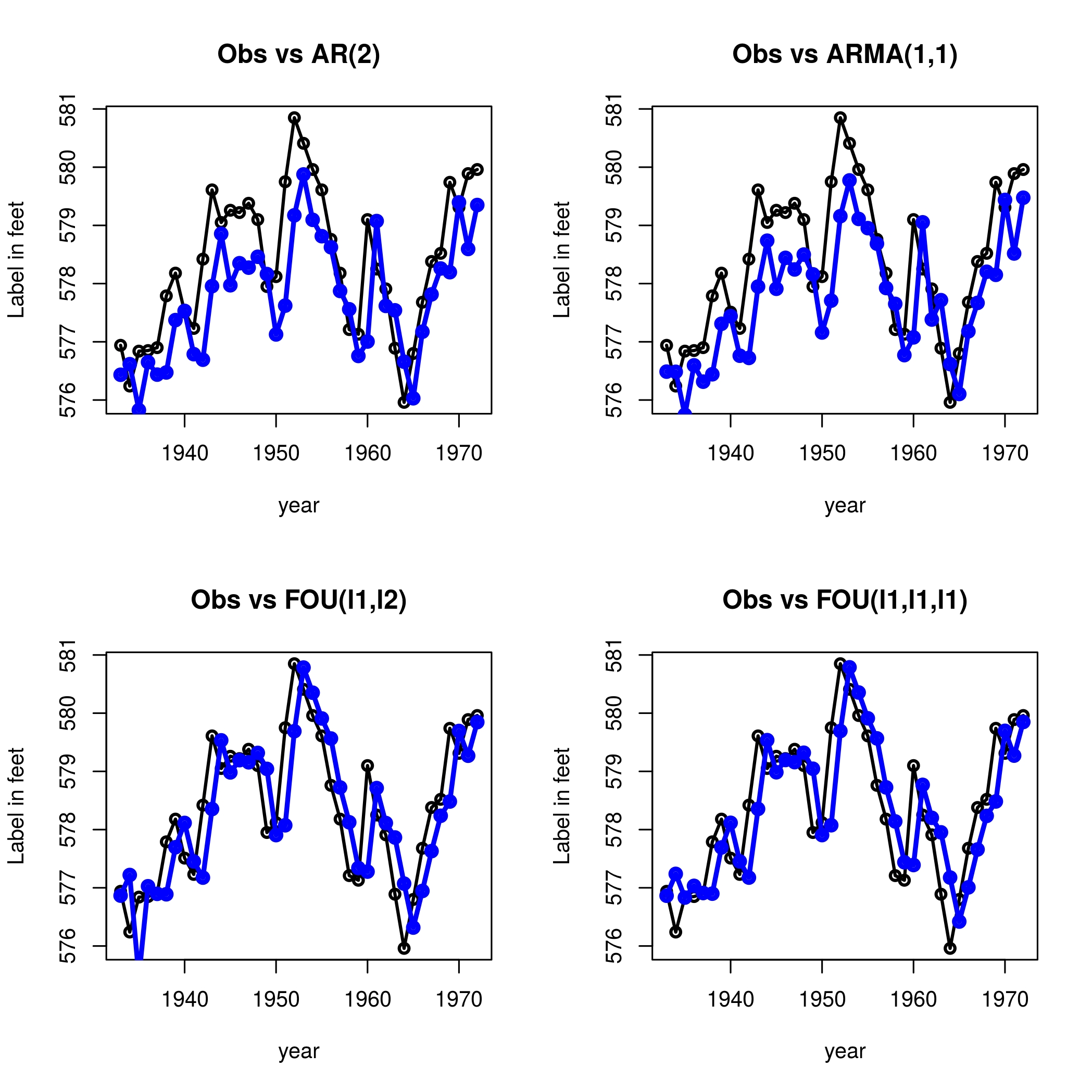} 
   \caption{Last 40 observed values (black) and corresponding predictions (blue) according to the
   adjusted model for the Lake Huron
   data set.} 
  \label{Huron40}
\end{figure}

In Table \ref{tableHuron}, we show the values of $W_2$, $RMSE$, $W_{1}$  and $MAE$, for the 
adjusted AR$(2)$, ARMA$(1,1)$ and different FOU$(p)$ for $p=2,3,4$
models adjusted for $T=30$, $H=0.5$ and $\sigma=1$.

We see that the performances of all models considered are similar.
We see that 
$\text{FOU}\left( \lambda^{(3)},\sigma ,H\right)$ model obtains slightly better results, and is
clearly better than the AR$(2)$ and ARMA$(1,1)$ models.

\begin{table}\caption{Values of  $W_2$, $RMSE$, $W_{1}$  and $MAE$ for different models, adjusted to the series 
 ``level in feet, Lake Huron'', for $m=40$ predictions and $T=30$, $H=0.5$, $\sigma=1$.
}
\label{tableHuron}
\centering
 
$%
\begin{array}{|c|cccc|}
\hline
\text{Model} & W_2 & RMSE & W_{1} & MAE \\ 
\hline
\text{AR}\left( 2\right) & 0.8421 & 0.8961 & 0.6345 & 0.7262 \\ 
\text{ARMA}\left( 1,1\right) & 0.8426 & 0.8999 & 0.6322 & 0.7271 \\ 
\text{FOU}\left( \lambda _{1},\lambda _{2},\sigma ,H\right) & 0.8850  & 0.7877
& 0.6903 & 0.6361 \\ 
\text{FOU}\left( \lambda _{1},\lambda _{2},\lambda _{3},\sigma ,H\right) & 0.8862
 & 0.7569  & 0.6967  & \textbf{0.6061}  \\ 
\text{FOU}\left( \lambda _{1},\lambda _{2},\lambda _{3},\lambda _{4},\sigma
,H\right) &0.8862 & \textbf{0.7568}  & 0.6966  & \textbf{0.6061}  \\ 
\text{FOU}\left( \lambda ^{\left( 2\right) },\sigma ,H\right) & 0.8788  & 0.7834
 & 0.6919 & 0.6192  \\ 
\text{FOU}\left( \lambda ^{\left( 3\right) },\sigma ,H\right) & \textbf{0.8867} & \textbf{0.7568}
 & \textbf{0.6973} & 0.6062\\ 
\text{FOU}\left( \lambda ^{\left( 4\right) },\sigma ,H\right) & 0.8852 & 0.7572
& 0.6939  & 0.6086 \\
\hline
\end{array}%
$
\end{table}

To have an idea of how these values can be changed for different values of 
the number $m$ of predictions, 
in Figure
\ref{Willmott_MAE} we show the results of the 
Willmott Index ($W_2$) and MAE for values of $m$ between $10$ to $40$.

\begin{figure}[H]
 \centering
   \includegraphics[scale=0.382]{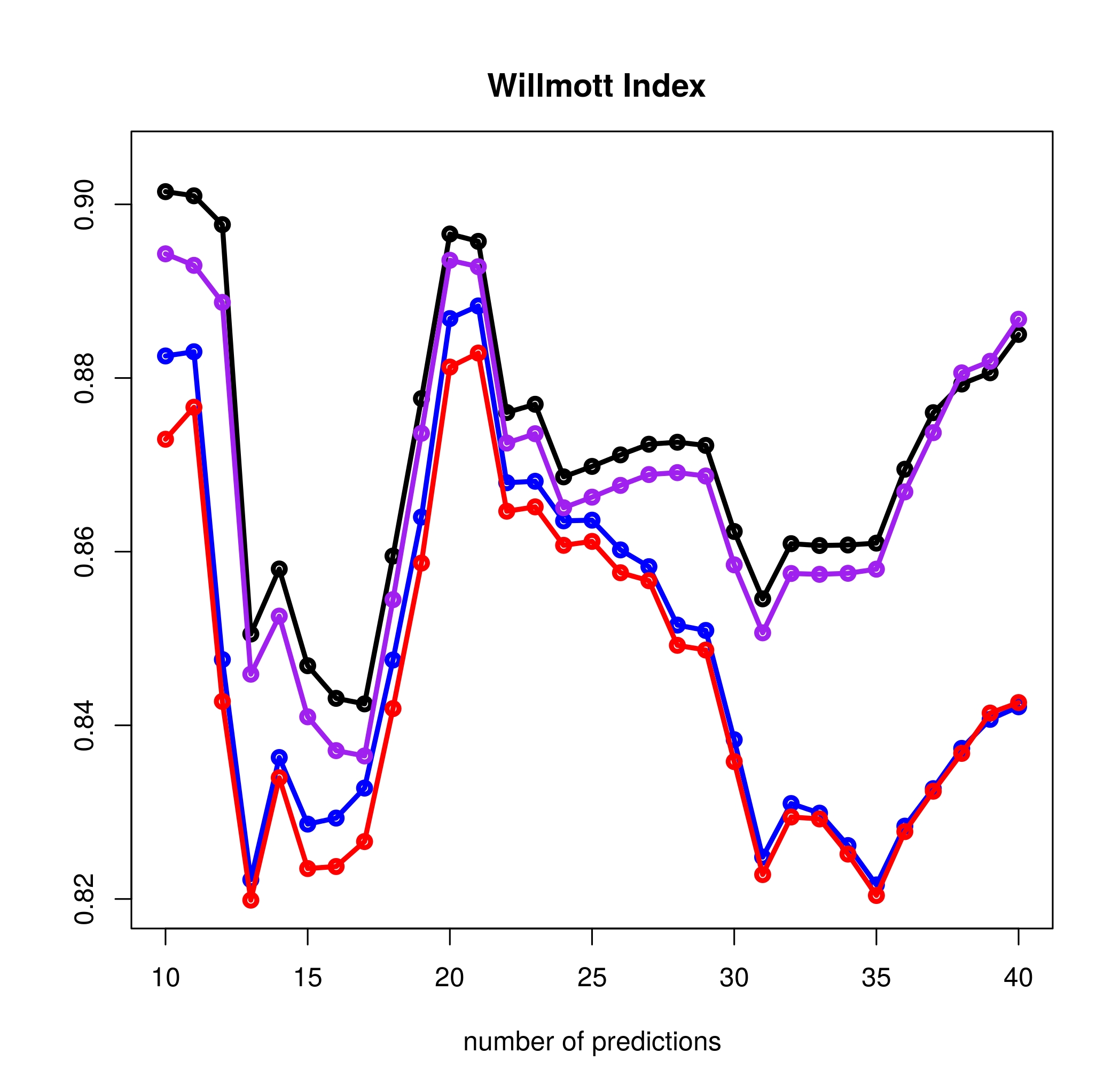} 
   \includegraphics[scale=0.382]{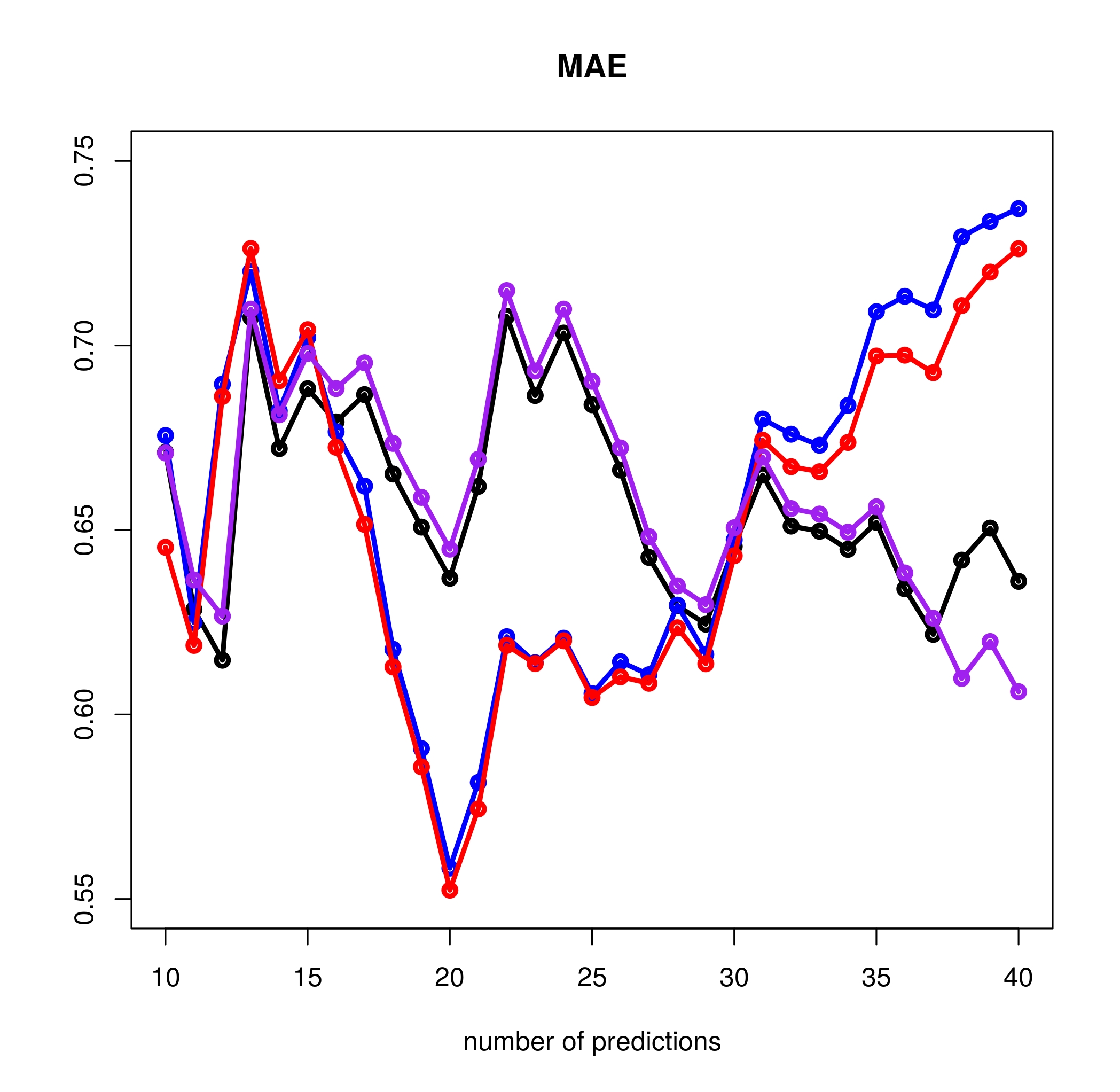} 
\caption{Willmott Index ($W_2$) on left and MAE on right, for the adjusted
FOU$(\lambda_{1},\lambda_{2},\sigma,H )$ (black),  FOU$(\lambda^{(3)},\sigma, H)$ (purple),  ARMA$(1,1)$ (blue) and 
    AR$(2)$ (red) for $m=10,11,....,40$ predictions for the Lake Huron data set.} 
  \label{Willmott_MAE}
\end{figure}
Figure \ref{Willmott_MAE} shows that for both measures of the quality of prediction, the FOU
models clearly outperform
AR$(2)$ and ARMA$(1,1)$  as the number of predictions ($m$) grows.

\subsection{Affluent energy generated by hydroelectric dams in Uruguay}
In this case we study a time series very different from the ones  previously considered.
On the one hand, it is longer than them. On 
the other hand  \cite{affluent} and \cite{represas} show that  this time series 
presents a long memory behaviour. 
We start with the weekly data set of affluent energy generated
by hydroelectric dams in Uruguay between the first week of 1909 and the  
last week of 2012. The observations present a seasonal component that  was removed. 
In \cite{affluent} there can be found a more detailed 
description of this time series together with a comparison, taking into account
the predictive power, of an ARFIMA$(p,d,q)$ model with an FOU$(\lambda,\sigma,H)$ 
model. In our case the time series has a sample size of length $5408$, and 
we have adjusted FOU$(p)$ models for $p=2,3,4$ with the first
$5148$ terms, and we have predicted the next $260$ weeks (this corresponds to the
2009 to 2012). In each case, we have made a one-step prediction (to predict a future
 value $X_t$ we have estimated the parameters using the information from the 
 time series for all times earlier
than $t$). In \cite{affluent}, an ARFIMA$(1,d,3)$ is proposed to adjust the observations.
The performance of the FOU$(p)$ models at estimating the value of $\sigma$ or using $\sigma=1$ 
are
similar, therefore we prefer to take $\sigma=1$ and we have chosen the value of $T$ as 
suggested in 
subsection 5.1. In Table \ref{table_parameters_affluent} we show the values of $T$ and the 
estimate
of $\lambda$ for each model considered. The estimation of $H$ yielded $\hat{H}=0.7114.$ Although 
the optimal value of $T$ considered is different for each adjusted FOU$(p)$ model, there are little 
variations for different values of $T$. 

\begin{table}[h!]
\caption{Values of $T$ considered and values of $\hat{\lambda}$  for different
FOU$(p)$ adjusted models.}
\begin{small}$\begin{array}{|c|ccccc|}
 \hline
 &\text{FOU}\left( \lambda ^{\left( 2\right) } ,H\right)
 &\text{FOU}\left( \lambda ^{\left( 3 \right) } ,H\right) 
 &\text{FOU}\left( \lambda ^{\left( 4\right) } ,H\right) 
 &\text{FOU}\left( \lambda_1,\lambda_2  ,H\right)&
  \text{FOU}\left( \lambda_1,\lambda_2,\lambda_3  ,H \right )  \\
  \hline
  $T$ & 280 & 230 & 290 &200 &300 \\
  \hat{\lambda} & 0.3115 & 0.3595 & 0.3617 & (0.5668,0.5802) & (0.3062,0.3162,0.3267) \\
  
 \hline
\end{array}
$ \end{small}
 
 \label{table_parameters_affluent}
\end{table}
In Table \ref{tablearfima} we show a comparison, in terms of for the four measures of quality
of prediction of the
FOU$(p)$ models and  ARFIMA$(1,d,3)$ and we include the ARFIMA$(3,d,1)$ model because it was
the model that achieved the best predictions for the latest four years (2009 to 2012) 
from among the ARFIMA$(p.d.q)$
for values of $p,q \in \{0,1,2,3\}$.

\begin{table}[h!]\caption{ Values of  $W_2$, $RMSE$, $W_{1}$  and $MAE$ for different models, 
adjusted to the non--stationary
series `affluent energy generated by hydroelectric dams', for  the years 2009 to 2012 ($\sigma=1$).}
 \centering

\begin{scriptsize}
$%
\begin{array}{|c|cccc|c|cccc|}
\hline
\text{2009} & W_2 & RMSE & W_{1} & MAE & 2010 & W_2 & RMSE & W_{1} & MAE  \\ 
\hline
\text{ARFIMA}(3,d,1) & 0.8184 & \textbf{0.3622} & 0.5857 & 0.2853  &  & 0.9497 & 0.5074 & 0.7943 & 0.3992  \\ 
\text{ARFIMA}(1,d,3) & 0.8137 & 0.3605 & 0.5827 & 0.2838  &  & 0.9497 & 0.5076 & 0.7900 & 0.4077   \\ 
\text{FOU}\left( \lambda _{1},\lambda _{2} ,H\right) &0.8418  &0.3782  &0.6499  &0.2699  &  &0.9606  &0.4940  & 0.8353 &0.3507  \\ 
\text{FOU}\left( \lambda _{1},\lambda _{2},\lambda _{3} ,H\right) &0.8422  &0.3776  &0.6510  &0.2691  &  &0.9611  &0.4903  &0.8365  &0.3474   \\ 

\text{FOU}\left( \lambda ^{\left( 2\right) } ,H\right) & 0.7385  & 0.3929  & 0.5168  & 0.3089  &  & 0.9470  & 0.5196  & 0.7934  & 0.4009    \\ 
\text{FOU}\left( \lambda ^{\left( 3\right) } ,H\right) & 0.8423  & 0.3775  & 0.6511  & 0.2690  & & \textbf{0.9611 }& \textbf{0.4896 } & \textbf{0.8367}  & \textbf{0.3465}   \\ 
\text{FOU}\left( \lambda ^{\left( 4\right) } ,H\right) &\textbf{0.8445} & 0.3755  & \textbf{0.6619}  & \textbf{0.2619 } &  &0.9610  & 0.4903  & 0.8363  & 0.3476   \\
\hline

\text{2011} & W_2 & RMSE & W_{1} & MAE & 2012 & W_2 & RMSE & W_{1} & MAE  \\ 
\hline
\text{ARFIMA}(3,d,1) & 0.9270 & \textbf{0.6617} & \textbf{0.7668} & \textbf{0.4841}  &  & \textbf{0.7406} & \textbf{0.5161} & \textbf{0.5549} & \textbf{0.3980}  \\ 
\text{ARFIMA}(1,d,3) & 0.9245 & 0.6684 & 0.7609 & 0.5074  &  & 0.7293 & 0.5165 & 0.5479 & 0.3993  \\ 
\text{FOU}\left( \lambda _{1},\lambda _{2} ,H\right) & 0.9291 &0.7276  &0.7611  &0.5559  &  &0.7004  &0.6153  &0.5494  &0.4535  \\ 
\text{FOU}\left( \lambda _{1},\lambda _{2},\lambda _{3} ,H\right) &0.9296  &0.7242  &0.7624  &0.5523  &  &0.7003  &0.6147  &0.5499  &0.4521   \\ 

\text{FOU}\left( \lambda ^{\left( 2\right) } ,H\right) &0.9002 &0.7786  &0.7047  &0.6087  &  &0.6298  &0.5393  &0.4623  &0.4192    \\ 
\text{FOU}\left( \lambda ^{\left( 3\right) } ,H\right) &0.9296  &0.7224  &0.7626  &0.5499  &  &0.7004  &0.6139  &0.5498  &0.4517   \\ 
\text{FOU}\left( \lambda ^{\left( 4\right) } ,H\right) &\textbf{0.9300}  &0.7179  &0.7636  &0.5459  &  &0.6994  &0.6142  &0.5490  &0.4513   \\
\hline
\end{array}%
$ \end{scriptsize}
\label{tablearfima}
\end{table}
From Table \ref{tablearfima} we can deduce that in 2009 and 2010, the FOU$(\lambda^{(4)},H)$
model has the best performance.
In addition, in 2011 and 2012, the ARFIMA$(3,d,1)$ has the best performance.
In general, in the latest four years, performances of ARFIMA$(3,d,1)$ and FOU$(\lambda^{(4)},H)$
are similar with a slight advantage for  FOU$(\lambda^{(4)},H)$ in  Willmott's index 
and a  slight advantage for ARFIMA$(3,d,1)$ in the RMSE and MAE measures. Like the example of Series A,
the FOU$(p)$ models obtain better results in  Willmott's index.

\section{Remarks}
\begin{enumerate}
\item In this paper we has shown that the FOU$(p)$ processes can be used to model a wide 
range of time series varying from short range dependence to long range dependence.
\item Another advantage to using an FOU$(p)$ process to model a continuous time data set 
(instead a discrete time model) is that it has
a parameter ($H$) that gives a measure of the irregularity of the trajectories in the process.
When the time series is in continuous time, the parameter $H$ can give extra 
information about how irregular the trajectories are.
\item We have suggested how to obtain a suitable value of $T$ in order to optimize the criterion
that is required by the investigator.
\item In the three real data sets considered, we observe results as shown in Figure \ref{fou_lambda2},
that is, little variation between the measures considered in some range of values of $T$ and an 
abrupt change in the performance out of this range (for example, from $T=5$ to $T=6$ in Figure \ref{fou_lambda2}).
 \item In the cases studied in this paper we have shown that it is possible to avoid the 
 estimation of the 
 parameter $\sigma$ (just taking $\sigma=1$). In general, we can estimate $\sigma$ using the 
 method
 proposed in subsection 3.1.
\end{enumerate}

\section{Conclusions}

In this paper we have presented a way to model an observed time series by a FOU process. We  have shown that the FOU$(p)$ processes can be used  to model a wide range of time series
varying from short range dependence to long range dependence, with performance similar 
to the ARMA or ARFIMA
models and in several cases outperforming them.
From theoretical point of view, we have presented a way to extend the theoretical results for FOU$(p)$
processes for $H<1/2$. We have presented a method to estimate
the  $\lambda$ parameters when the process is observed on a discretized and 
equispaced set in an  interval $[0,T]$ 
and  proved, under some conditions on $T$ and the sample size $n$, that
this procedure gives  a consistent estimator.   By simulations we have corroborated
the consistency of all the estimators of the parameters. Lastly we have given a complete description of 
FOU$(2)$ and FOU$(3)$ processes by an explicit formula for their auto-covariance functions,
and showed how these explicit formulas can be obtained for the FOU$(p)$ processes for $p \geq 4$.

\section{Proofs}
To prove Theorem \ref{consistency_theorem}, we need the following  lemmas.

\begin{lemma}
	Under the conditions of Theorem \ref{consistency_theorem} and fixed $\sigma$ and $H$, 
	then, given $\varepsilon >0$
	there exists a $T_{0}>0$ and a random variable $K_{\varepsilon} $ such that

	\begin{equation*}
	P\left( \sup_{\lambda \in \Lambda }\left\vert U_{T}\left( \lambda ,\sigma
	,H\right) -U_{T^{\prime }}\left( \lambda ,\sigma ,H\right) \right\vert \leq
	K_{\varepsilon}\left\vert T-T^{\prime }\right\vert \text{ }\right) \geq 1-\varepsilon
	\ \ \text{	for all } T,T' \geq T_0. 
	\end{equation*} 
	\label{lemmaU_T}
\end{lemma}
\begin{remark}
	Lemma \ref{lemmaU_T} together with the completeness of the space of continuous
functions defined on a compact $\Lambda $, allows  deducing the existence of
a function $U\left( \lambda ,\sigma ,H\right) $ such that for any fixed $%
\sigma $ and $H$, the $U_{T}\left( .,\sigma ,H\right) $ converge uniformly on $%
\Lambda $ to $U\left( .,\sigma ,H\right) $ in probability as $T\rightarrow
+\infty .$ \label{remark_UT}
\end{remark}
\begin{proof}[Proof of Lemma \ref{lemmaU_T}]
	\begin{equation*}
	I_{T}(x)=\frac{1}{2\pi T}\left( \int_{0}^{T}X_{t}\cos \left( tx\right) dt\right)
	^{2}+\frac{1}{2 \pi T}\left( \int_{0}^{T}X_{t}\sin \left( tx\right) dt\right)
	^{2}:=I_{T}^{\left( 1\right) }(x)+I_{T}^{\left( 2\right) }(x).
	\end{equation*}%
	\begin{equation*}
	\frac{\partial I_{T}^{\left( 1\right) }(x)}{\partial T}=\frac{2TX_{T}\cos
		\left( Tx\right) \int_{0}^{T}X_{t}\cos \left( tx\right) dt-\left(
		\int_{0}^{T}X_{t}\cos \left( tx\right) dt\right) ^{2}}{2\pi T^{2}}.
	\end{equation*}%
	Then, 
	\begin{equation*}
	\left\vert \frac{\partial I_{T}^{\left( 1\right) }(x)}{\partial T}%
	\right\vert \leq \frac{2\left\vert X_{T}\right\vert \left\vert
		\int_{0}^{T}X_{t}\cos \left( tx\right) dt\right\vert }{2\pi T}+\frac{1}{2\pi}\left( \frac{1}{T}%
	\int_{0}^{T}X_{t}\cos \left( tx\right) dt\right) ^{2}.
	\end{equation*}%
	Observe that 
	\begin{equation*}
	\frac{\left\vert \int_{0}^{T}X_{t}\cos \left( tx\right) dt\right\vert }{T}%
	\leq \frac{\int_{0}^{T}\left\vert X_{t}\right\vert dt}{T}.
	\end{equation*}%
	On the one hand, 
	\begin{equation}\label{ergodic}
	\frac{\int_{0}^{T}\left\vert X_{t}\right\vert dt}{T}\overset{%
		a.s.}{\rightarrow }\mathbb{E}\left( \left\vert X_{0}\right\vert \right) \text{ as } 
	T\rightarrow +\infty \end{equation}
	                       by the ergodic theorem. On the other hand, observing
	that the distribution of $X_{T}$ does not depend on $T,$ and $\frac{%
		\left\vert \int_{0}^{T}X_{t}\cos \left( tx\right) dt\right\vert }{T}\leq 
	\frac{\int_{0}^{T}\left\vert X_{t}\right\vert dt}{T}$, we obtain that there exists 
	$T_0$ and a random variable $K_{\varepsilon}$ such that $\frac{%
		2\left\vert X_{T}\right\vert \left\vert \int_{0}^{T}X_{t}\cos \left(
		tx\right) dt\right\vert }{T}$ is bounded for all $T \geq T_0$ (and for all $x$). Then 
		for all $T \geq T_0$,  such that
		$P \left ( \left \vert \frac{\partial I_{T}^{\left( 1\right) }(x)}{%
		\partial T} \right \vert \leq K_{\varepsilon} \text{ for all } x \right ) \geq 1-\varepsilon$.
	
	Proceeding analogously with $\frac{\partial I_{T}^{\left( 2\right) }(x)}{%
		\partial T}$, we obtain that given $\varepsilon >0,$ there exists a $T_{0}$ and a random variable 
		$K_{\varepsilon}$
		such that   
	\begin{equation}
P\left( \left\vert \frac{\partial I_{T}(x)}{%
	\partial T}\right\vert \leq K_{\varepsilon} \text{ for all} \ x
	\right) \geq 1-\varepsilon 
	\label{132}
	\end{equation} 
	for all $T \geq T_0.$
	Therefore, if $\left\vert \frac{\partial I_{T}(x)}{%
	\partial T}\right\vert \leq K_{\varepsilon}$ for all $ x$, then from the mean value theorem
	we obtain that
	\begin{equation*}
	\left\vert U_{T}\left( \lambda ,\sigma ,H\right) -U_{T^{\prime }}\left(
	\lambda ,\sigma ,H\right) \right\vert \leq \int_{0}^{+\infty }\frac{%
		\left\vert I_{T}(x)-I_{T^{\prime }}(x)\right\vert }{f^{\left( X\right)
		}\left( x,\lambda \right) }w(x)dx\leq 
	\end{equation*}	
	\begin{equation*}	K_{\varepsilon}\left\vert T-T^{\prime }\right\vert
	\int_{0}^{+\infty }\frac{w(x)}{f^{\left( X\right) }\left( x,\lambda \right) }%
	dx,
	\end{equation*}%
	for all $T,T^{\prime }\geq T_{0}.$ 
	Thus, using that $\int_{0}^{+\infty }\frac{w(x)}{f^{\left( X\right) }\left(
		x,\lambda \right) }dx$ is bounded on the compact $\Lambda $ and renaming $K_{\varepsilon}$%
	, we obtain that 
	\begin{equation*}
	\sup_{\lambda \in \Lambda }\left\vert U_{T}\left( \lambda ,\sigma ,H\right)
	-U_{T^{\prime }}\left( \lambda ,\sigma ,H\right) \right\vert \leq
	K_{\varepsilon}\left\vert T-T^{\prime }\right\vert .
	\end{equation*}%

\end{proof}

\begin{lemma}
Under the conditions of Theorem \ref{consistency_theorem}, then, given $\varepsilon >0$
there exists a $T_0$ and a random variable $K_\varepsilon$ such that the probability of
\begin{equation*}
\max \left \{ \left\vert \frac{\partial h_{T} }{\partial
	x}\left( x,\lambda ,\sigma ,H\right)\right\vert, \left\vert
 \frac{\partial h_{T} }{\partial
	\sigma}\left( x,\lambda ,\sigma ,H\right)\right\vert, \left\vert \frac{\partial h_{T} }{\partial
	H}\left( x,\lambda ,\sigma ,H\right)\right\vert \right \}
 \leq K_\varepsilon T^{2}  
\end{equation*}%
for all $\left( x,\lambda ,\sigma ,H\right) \in \left[ 0,T\right] \times
\Lambda \times \left[ \sigma_1,\sigma_2\right] \times \left[ h_{1},h_{2}\right],$ 
is greater than or equal to $1-\varepsilon$.
\label{lemma_h_bounded}
\end{lemma}
\begin{proof}[Proof of Lemma \ref{lemma_h_bounded}]
Firstly, we will prove that there exists a $T_0 >0$ and a random variables $K_\varepsilon'$,
$K_\varepsilon ''$ such that 
\begin{equation}\label{Ix}
P \left ( \left \vert I_T(x) \right \vert \leq K_\varepsilon' T 
 \text{ for all } x \right ) \geq 1- \varepsilon/2 \text{ for all } T \geq T_0
                            \end{equation}
and 
\begin{equation}\label{I'respect_x}
P \left ( \left \vert \frac{\partial I_{T} }{\partial x}(x) \right \vert \leq K_\varepsilon'' T^{2} 
 \text{ for all } x \right ) \geq 1- \varepsilon/2 \text{ for all } T \geq T_0.
                            \end{equation}
    Observe that
$$
I_T(x)=\frac{1}{2\pi T}\left \vert \int_0^{T} X_t e^{itx}dt\right \vert=
\frac{T}{2 \pi}\left \vert \frac{1}{T}\int_0^{T}X_te^{itx}dt \right \vert ^{2} \leq 
\frac{T}{2 \pi} \left ( \frac{1}{T}\int_0^{T}\left \vert X_t \right \vert dt \right ) ^{2}
$$
and applying the ergodic theorem to the process $\{|X_t|\}_{t \in \mathbb{R}}$, we deduce (\ref{Ix}).\\                        
                            
 Using that $t \in [0,T]$ and applying the Cauchy--Schwartz inequality, we obtain that                          
$$
\left \vert \frac{\partial I_{T} }{\partial x}(x) \right \vert =$$   $$\frac{1}{\pi T}
\left \vert \int_0^{T}X_t \cos(tx)dt \int_0^{T}tX_t \sin(tx)dt + 
\int_0^{T}X_t \sin(tx)dt \int_0^{T}tX_t \cos(tx)dt\right \vert  \leq
$$ 
\begin{equation}\label{101}
\frac{2}{\pi}   \left ( \int_0^{T} \vert X_t \vert dt\right )^{2} 
\leq \frac{2T}{\pi} \int_0^{T}X_t^{2}dt.
\end{equation}
Applying the ergodic theorem to the process $\{X_t^{2}\}_{t \in \mathbb{R}}$, we deduce (\ref{I'respect_x}).\\

Consider $\widetilde{\lambda } \in \Lambda$ such that $\lambda _{i}\leq \widetilde{%
	\lambda }_{i}$ for all $i=1,2,...,q.$ Then 

\begin{equation*}
\left\vert \frac{\partial h_{T} }{\partial
	\sigma }\left( x,\lambda ,\sigma ,H\right)\right\vert =\left\vert \frac{2}{\sigma }-\frac{2I_{T}(x)2 \pi%
	\prod_{i=1}^{q}\left( \lambda_{i}^{2}+x^{2}\right) ^{q_{i}}}{%
	\sigma ^{3}\Gamma \left( 2H+1\right) \sin \left( H\pi \right) x^{2p-1-2H}}%
\right\vert \frac{w(x)}{2\pi}\leq 
\end{equation*}%
\begin{equation}
\left( \frac{2}{\sigma_1}+\frac{c_{1}I_T (x)\prod_{i=1}^{q}\left( \widetilde{\lambda }%
	_{i}^{2}+x^{2}\right) ^{q_{i}}}{\sigma_1^{3}x^{2p-1-2H}}\right) \frac{w(x)}{2 \pi} \  \text{where} \ c_1 \text{
	is a constant}. \label{derive_h_t}
\end{equation}%
Condition $a\geq 2p$ arranges that (\ref{derive_h_t}) has no singularity in $x=0.$ On the
other hand, condition $b\geq a+3$ arranges that  (\ref{derive_h_t}) goes to zero as  $%
x\rightarrow +\infty $. Therefore, from (\ref{Ix}) we have that 
there exists a $T_0>0$ and a random variable $K_\varepsilon^{(1)}$ such that
the probability of  $%
\left\vert \frac{\partial h_{T} }{\partial
	\sigma }\left( x,\lambda ,\sigma ,H\right)\right\vert \leq K_\varepsilon^{(1)} T$ for all 
	$\left( x,\lambda ,\sigma ,H\right) \in %
\left[ 0,T\right] \times \Lambda \times \left[ \sigma_1,\sigma_2 \right] \times \left[
h_{1},h_{2}\right]$ is greater than or equal to $1-\varepsilon/2$ for all $T \geq T_0$.
 $$\left\vert \frac{\partial h_{T} }{\partial H}\left( x,\lambda
,\sigma ,H\right)\right\vert \leq %
\left\vert \frac{2\Gamma ^{\prime }\left( 2H+1\right) }{\Gamma \left(
	2H+1\right) }+\frac{\pi \cos \left( H\pi \right) }{\sin \left( H\pi \right)}-2\log x%
\right \vert w(x)+$$ 
$$\left \vert \frac{I_{T}(x)\prod_{i=1}^{q}\left( \lambda _{i}^{2}+x^{2}\right) ^{q_{i}}%
}{\sigma ^{2}x^{2p-1-2H}}\frac{2\Gamma ^{\prime }\left( 2H+1\right) \sin
	\left( H\pi \right) +\pi \Gamma \left( 2H+1\right) \cos \left( H\pi \right) 
}{\left( \Gamma \left( 2H+1\right) \sin \left( H\pi \right) \right) ^{2}}%
\right\vert w(x)+$$

$$\left \vert \frac{I_{T}(x)\prod_{i=1}^{q}\left( \lambda _{i}^{2}+x^{2}\right) ^{q_{i}}%
}{\sigma ^{2}x^{2p-1-2H}}\frac{2 \log x 
}{ \Gamma \left( 2H+1\right) \sin \left( H\pi \right) }%
\right\vert w(x).
$$

Analogously to the \ previous case we obtain that there exists a random variable 
$K_\varepsilon^{(2)}$ such that  the probability of $\left\vert \frac{%
	\partial h_{T} }{\partial H}\left( x,\lambda ,\sigma ,H\right)\right\vert
\leq K_\varepsilon^{(2)} T $ for all $\left( x,\lambda ,\sigma ,H\right) \in %
\left[ 0,T\right] \times \Lambda \times \left[ \sigma_1,\sigma_2 \right] \times \left[
h_{1},h_{2}\right]$ is greater than or equal to $1-\varepsilon/2$ for all $T \geq T_0$.

We call $h_{T}\left( x,\lambda ,\sigma ,H\right) =g_{T}\left( x,\lambda
,\sigma ,H\right) \frac{w(x)}{2\pi }$ where 
\begin{equation*}
g_{T}\left( x,\lambda ,\sigma ,H\right) =\ln \left( \frac{\sigma ^{2}\Gamma
	\left( 2H+1\right) \sin \left( H\pi \right) x^{2p-1-2H}}{p_{2p}(x)}\right) +%
\frac{I_{T}(x)p_{2p}(x)}{\sigma ^{2}\Gamma \left( 2H+1\right) \sin \left(
	H\pi \right) x^{2p-1-2H}}
\end{equation*}%
and $p_{2p}(x)=2 \pi \prod_{i=1}^{q}\left( \lambda _{i}^{2}+x^{2}\right) ^{q_{i}}$
is a polynomial of order $2p.$

Then $\left \vert \frac{\partial g_{T} }{\partial x}\left( x,\lambda ,\sigma ,H\right)%
\frac{w(x)}{2\pi } \right \vert \leq $ 
\begin{equation*}
\left \vert \frac{2p-1-2H}{x}-\frac{p_{2p}^{\prime }(x)}{p_{2p}(x)} \right \vert 
\frac{w(x)}{2\pi}+ 
\end{equation*}
\begin{equation*}
\left \vert 
\frac{\left( 
	\frac{\partial I_{T}(x)}{\partial x}p_{2p}(x)+I_{T}(x)p_{2p}^{\prime
	}(x)\right) x^{2p-1-2H}-\left( 2p-1-2H\right) x^{2p-2-2H}I_{T}(x)p_{2p}(x)}{%
	\sigma ^{2}\Gamma \left( 2H+1\right) \sin \left( H\pi \right) x^{4p-2-4H}}%
\right \vert  \frac{w(x)}{2\pi }
\end{equation*}%
Analogously to the case $\frac{\partial g_{T} }{\partial x}\left( x,\lambda ,\sigma
,H\right)$, observe  that conditions $b\geq a+3,$ $a\geq 2p$ allows to affirm
that (in a set of prabability greater than or equal to $1-\varepsilon/2$)  that
there exists a random variables $K_\varepsilon^{(3)}$, $K_\varepsilon^{(4)}$
such that 
$
\left\vert \frac{\partial g_{T} }{\partial
	x}\left( x,\lambda ,\sigma ,H\right)\frac{w(x)}{2\pi }\right\vert \leq
	K_\varepsilon^{(3)} T^{2}.
$
and $
\left\vert g_{T}\left( x,\lambda ,\sigma ,H\right) \frac{w^{\prime }\left(
	x\right) }{2\pi }\right\vert \leq K_\varepsilon^{(4)} T^{2},
$ with probability greater than or equal to $1-\varepsilon/2$.
Therefore, there exists a $T_0>0$ and a random variable
$K_\varepsilon=\max\{K_\varepsilon^{(1)},K_\varepsilon^{(2)},K_\varepsilon^{(3)},K_\varepsilon^{(4)}\}$ 
such that  the probability of 
 \begin{equation*}
\max \left \{ \left\vert \frac{\partial h_{T} }{\partial
	x}\left( x,\lambda ,\sigma ,H\right)\right\vert, \left\vert
 \frac{\partial h_{T} }{\partial
	\sigma}\left( x,\lambda ,\sigma ,H\right)\right\vert, \left\vert \frac{\partial h_{T} }{\partial
	H}\left( x,\lambda ,\sigma ,H\right)\right\vert \right \}
 \leq K_\varepsilon T^{2}  
\end{equation*}%
for all $\left( x,\lambda ,\sigma ,H\right) \in \left[ 0,T\right] \times
\Lambda \times \left[ \sigma_1,\sigma_2\right] \times \left[ h_{1},h_{2}\right],$ 
is greater than or equal to $1-\varepsilon$ for all $T \geq T_0$.
\end{proof}

\begin{lemma}
Under the conditions of Theorem \ref{consistency_theorem}, we have
\begin{equation*}
\sup_{0\leq x\leq T_{n}}\left\vert I_{T_{n}}(x)-I_{T_{n}}^{\left( n\right)
}(x)\right\vert \overset{P}{\rightarrow }0\text{ as }n\rightarrow +\infty .
\end{equation*} \label{supI_T}
\end{lemma}
\begin{proof}[Proof of Lemma \ref{supI_T}]
	\[\]
For a fixed $T>0,$%
\begin{equation*}
\left\vert I_{T}(x)-I_{T}^{\left( n\right) }(x)\right\vert =\left\vert \frac{%
	1}{2\pi T}\left\vert \int_{0}^{T}e^{itx}X_{t}dt\right\vert ^{2}-\text{ }%
\frac{1}{2\pi T}\left\vert \frac{T}{n}\sum_{j=1}^{n}e^{\frac{ijTx}{n}}X_{%
	\frac{jT}{n}}\right\vert ^{2}\right\vert \leq 
\end{equation*}%
\begin{equation*}
\frac{1}{2\pi T}\left\vert \left( \int_{0}^{T}\cos \left( tx\right)
X_{t}dt\right) ^{2}-\text{ }\left( \frac{T}{n}\sum_{j=1}^{n}\cos \left( \frac{jTx}{n}%
\right) X_{\frac{jT}{n}}\right) ^{2} \right |+
\end{equation*}
$$\frac{1}{2\pi T} \left \vert \left( \int_{0}^{T}\sin \left( tx\right) X_{t}dt\right)
^{2}-\text{ }\left( \frac{T}{n}%
\sum_{j=1}^{n}\sin \left( \frac{jTx}{n}\right) X_{\frac{jT}{n}}\right)
^{2}\right\vert \leq $$

\begin{equation*}
\frac{1}{2\pi T}\left\vert \left( \int_{0}^{T}\cos \left( tx\right) X_{t}dt-%
\text{ }\frac{T}{n}\sum_{j=1}^{n}\cos \left( \frac{jTx}{n}\right) X_{\frac{jT%
	}{n}}\right) \left( \int_{0}^{T}\cos \left( tx\right) X_{t}dt+\text{ }\frac{T%
}{n}\sum_{j=1}^{n}\cos \left( \frac{jTx}{n}\right) X_{\frac{jT}{n}}\right)
\right\vert +
\end{equation*}%
\begin{equation*}
\frac{1}{2\pi T}\left\vert \left( \int_{0}^{T}\sin \left( tx\right) X_{t}dt-%
\text{ }\frac{T}{n}\sum_{j=1}^{n}\sin \left( \frac{jTx}{n}\right) X_{\frac{jT%
	}{n}}\right) \left( \int_{0}^{T}\sin \left( tx\right) X_{t}dt+\text{ }\frac{T%
}{n}\sum_{j=1}^{n}\sin \left( \frac{jTx}{n}\right) X_{\frac{jT}{n}}\right)
\right\vert =
\end{equation*}%
\begin{equation*}
I_{n,T}+I_{n,T}^{\prime }.
\end{equation*}%
On the one hand, 
\begin{equation*}
\left\vert \int_{0}^{T}\cos \left( tx\right) X_{t}dt+\text{ }\frac{T}{n}%
\sum_{j=1}^{n}\cos \left( \frac{jTx}{n}\right) X_{\frac{jT}{n}}\right\vert
\leq \int_{0}^{T}\left\vert X_{t}\right\vert dt+\text{ }\frac{T}{n}%
\sum_{j=1}^{n}\left\vert X_{\frac{jT}{n}}\right\vert \overset{a.s.}{%
	\rightarrow }2\int_{0}^{T}\left\vert X_{t}\right\vert dt.
\end{equation*}

On the other hand,

\bigskip 
\begin{equation*}
\left\vert \int_{0}^{T}\cos \left( tx\right) X_{t}dt-\text{ }\frac{T}{n}%
\sum_{j=1}^{n}\cos \left( \frac{jTx}{n}\right) X_{\frac{jT}{n}}\right\vert
=
\end{equation*}
$$\left\vert \sum_{j=1}^{n}\int_{(j-1)T/n}^{jT/n}\left( \cos \left( tx\right)
X_{t}-\cos \left( \frac{jTx}{n}\right) X_{\frac{jT}{n}}\right) dt\right\vert
\leq $$
\begin{equation}
\sum_{j=1}^{n}\int_{(j-1)T/n}^{jT/n}\left\vert \left( \cos \left( tx\right)
-\cos \left( \frac{jTx}{n}\right) \right) X_{t}\right\vert
dt+\sum_{j=1}^{n}\int_{(j-1)T/n}^{jT/n}\left\vert \cos \left( \frac{jTx}{n}%
\right) \left( X_{t}-X_{\frac{jT}{n}}\right) \right\vert dt. \label{cotacoseno}
\end{equation}

If $(j-1)T/n\leq t\leq jT/n$ and $0\leq x\leq T,$ then $\left\vert \cos
\left( tx\right) -\cos \left( jTx/n\right) \right\vert \leq \left\vert
\left( t-jT/n\right) x\right\vert \leq xT/n\leq T^{2}/n.$ Thus (\ref{cotacoseno}) is less
than or equal to%
\begin{equation}
\frac{T^{2}}{n}\int_{0}^{T}\left\vert X_{t}\right\vert
dt+\sum_{j=1}^{n}\int_{(j-1)T/n}^{jT/n}\left\vert X_{t}-X_{\frac{jT}{n}%
}\right\vert dt. \label{cotacoseno2}
\end{equation}%

Condition $T_n^{3}/n \rightarrow 0$ (see Remark \ref{T/n}) and replacing $T$ by $T_n$ allows to affirm that the 
first term in (\ref{cotacoseno2}) 
goes to zero in probability as $n\rightarrow +\infty
. $ To prove that the second term goes to zero too it is enough to prove
that $\sum_{j=1}^{n}\int_{(j-1)T/n}^{jT/n}\mathbb{E}\left( \left\vert
X_{t}-X_{\frac{jT}{n}}\right\vert \right) dt\rightarrow 0$ as $n\rightarrow
+\infty .$

\begin{equation}
\sum_{j=1}^{n}\int_{(j-1)T/n}^{jT/n}\mathbb{E}\left( \left\vert X_{t}-X_{%
	\frac{jT}{n}}\right\vert \right) dt\leq \sum_{j=1}^{n}\int_{(j-1)T/n}^{jT/n}%
\sqrt{\mathbb{E}\left( \left( X_{t}-X_{\frac{jT}{n}}\right) ^{2}\right) }dt. \label{variograma}
\end{equation}%
Using that the variogram of any FOU$\left( p\right) $ process satisfies the
equality $v(t)=\mathbb{E}\left( \left( X_{t}-X_{0}\right) ^{2}\right) =\frac{%
	\sigma ^{2}}{2}\left\vert t\right\vert ^{2H}+o\left( \left\vert t\right\vert
^{2H}\right) $ were $t \rightarrow 0$ (Theorem 3.2 of Kalemkerian \& Le\'{o}n), then 
there exists a constant $k$  such that $v(t) \leq k \vert t \vert ^{2H}$ for all $\vert t \vert \leq 1.$ Therefore,   $%
\mathbb{E}\left( \left( X_{t}-X_{\frac{jT}{n}}\right) ^{2}\right) =v\left(
t-jT/n\right) \leq k \left\vert t-jT/n\right\vert
^{2H} $ and  we obtain
that (\ref{variograma}) is less than or equal to

\bigskip 
\begin{equation*}
\sqrt{k}\sum_{j=1}^{n}\int_{(j-1)T/n}^{jT/n}\left(
-t+jT/n\right) ^{H}dt=
\frac{\sqrt{k} T^{H+1}}{\left( H+1\right) n^{H}}.
\end{equation*}
Condition $\frac{T_{n}^{H+1}}{n^{H}}\rightarrow 0$
(see Remark \ref{T/n}) and replacing $T$ by $T_{n}$, allows to
affirm that $I_{n,T_{n}}\overset{P}{\rightarrow }0.$ Analogously $%
I_{n,T_{n}}^{\prime }\overset{P}{\rightarrow }0$. 
\end{proof}

\begin{lemma}
 Under the conditions of Theorem \ref{consistency_theorem}, we have  
 \begin{equation*}
\sup_{\lambda \in \Lambda }\left\vert U_{T_n}^{\left( n\right) }\left( \lambda
,\widehat{\sigma },\widehat{H}\right) -U_{T_n}\left( \lambda ,\sigma
^{0},H^{0}\right) \right\vert \overset{P}{\rightarrow }0 \ \text{as} \  n \rightarrow +\infty.
\end{equation*}%
 
\end{lemma}
\begin{proof}[Proof of Lemma 4]\[\]

Firstly we will prove that for each $T>0$ there exists a random variable $M_{T}$ such that 
\begin{equation*}
\sup_{\lambda \in \Lambda }\left\vert U_{T}^{\left( n\right) }\left( \lambda
,\widehat{\sigma },\widehat{H}\right) -U_{T}\left( \lambda ,\sigma
^{0},H^{0}\right) \right\vert \leq TM_{T}\left( \frac{T}{2n}+\left\vert 
\widehat{\sigma }-\sigma ^{0}\right\vert +\left\vert \widehat{H}%
-H^{0}\right\vert \right) +A_{n,T} \label{sup_U_Tn}
\end{equation*}%
where $A_{n,T_n}\overset{P}{\rightarrow }0$ as $n \rightarrow +\infty.$

\begin{equation*}
\left\vert U_{T}^{\left( n\right) }\left( \lambda ,\widehat{\sigma },%
\widehat{H}\right) -U_{T}\left( \lambda ,\sigma ^{0},H^{0}\right)
\right\vert =\end{equation*}
\begin{equation*}
\left\vert \sum_{j=1}^{n}\int_{(j-1)T/n}^{jT/n}\left(
h_{T}\left( x,\lambda ,\sigma ^{0},H^{0}\right) -h_{T}^{\left( n\right)
}\left( jT/n,\lambda ,\widehat{\sigma },\widehat{H}\right) \right)
dx\right\vert \leq
\end{equation*}%
\begin{equation*}
\sum_{j=1}^{n}\int_{(j-1)T/n}^{jT/n}\left\vert h_{T}\left( jT/n,\lambda ,%
\widehat{\sigma },\widehat{H}\right) -h_{T}^{\left( n\right) }\left(
jT/n,\lambda ,\widehat{\sigma },\widehat{H}\right) \right\vert dx+
\end{equation*}%
\begin{equation*}
\sum_{j=1}^{n}\int_{(j-1)T/n}^{jT/n}\left\vert h_{T}\left( x,\lambda ,\sigma
^{0},H^{0}\right) -h_{T}\left( jT/n,\lambda ,\widehat{\sigma },\widehat{H}%
\right) \right\vert dx=A_{n,T}\left( \lambda \right) +B_{n,T}\left( \lambda
\right) .
\end{equation*}%
Observe that $h_{T}$ has bounded partial derivatives on $\left[ 0,T\right]
\times \Lambda \times \left[ a,b\right] \times \left[ h_1,h_2\right] $ and
using \ the mean value theorem, there exists a random variable $M_{T}$ such that
for any $\lambda \in \Lambda $

\begin{equation*}
B_{n,T}\left( \lambda \right) \leq M_{T}\sum_{j=1}^{n}\int_{(j-1)T/n}^{jT/n}
\left[ jT/n-x+\left\vert \widehat{\sigma }-\sigma ^{0}\right\vert
+\left\vert \widehat{H}-H^{0}\right\vert \right] dx=\end{equation*}
\begin{equation*}
M_{T}\left( \frac{T^{2}}{%
	2n}+T\left\vert \widehat{\sigma }-\sigma ^{0}\right\vert +T\left\vert 
\widehat{H}-H^{0}\right\vert \right) .
\end{equation*}

From Lemma 2, we have $M_T \leq K_T T^{2}$, and from condition $\frac{T_n^{4}}{n} \rightarrow 0$
(see Remark \ref{T/n}) and replacing $T$ by $T_n$ allows to affirm that
$\sup_{\lambda \in \Lambda}B_{n,T_n}(\lambda)\overset{P}{\rightarrow }0.$

Now, if we define $A_{n,T}=\sup_{\lambda \in \Lambda }A_{n,T}\left( \lambda
\right) ,$ we will prove that $A_{n,T_n}\overset{P}{\rightarrow }0.$

We call $M_{n,T}=\sup_{0\leq x\leq T}\left\vert I_{T}(x)-I_{T}^{\left(
	n\right) }(x)\right\vert $, then $$\left\vert h_{T}\left( x,\lambda ,\sigma ,H\right) -h_{T}^{\left( n\right)
}\left( x,\lambda ,\sigma ,H\right) \right\vert =$$

\begin{equation*}
\left\vert
I_{T}(x)-I_{T}^{\left( n\right) }(x)\right\vert \frac{w(x)}{f^{\left(
		X\right) }\left( x,\lambda ,\sigma ,H\right) }\leq \frac{M_{n,T}w(x)}{%
	f^{\left( X\right) }\left( x,\lambda ,\sigma ,H\right) }.
\end{equation*}%
Thus 
\begin{equation*}
A_{n,T}\left( \lambda \right) =\frac{T}{n}\sum_{j=1}^{n}\left\vert
h_{T}\left( jT/n,\lambda ,\widehat{\sigma },\widehat{H}\right)
-h_{T}^{\left( n\right) }\left( jT/n,\lambda ,\widehat{\sigma },\widehat{H}%
\right) \right\vert dx\leq 
\end{equation*}
$$M_{n,T}\frac{T}{n}\sum_{j=1}^{n}\frac{w\left(
	jT/n\right) }{f^{\left( X\right) }\left( jT/n,\lambda ,\widehat{\sigma },%
	\widehat{H}\right) }.$$
Observe that there exists $\widetilde{\lambda }\in \Lambda $ such that $%
\prod_{i=1}^{q}\left( \lambda _{i}^{2}+x^{2}\right) ^{p_{i}}\leq
\prod_{i=1}^{q}\left( \widetilde{\lambda }_{i}^{2}+x^{2}\right) ^{p_{i}},$
and  there exists a constant $k$ such that 
\begin{equation}
\frac{1}{f^{\left( X\right) }\left( x,\lambda ,\sigma ,H\right) }=\frac{%
	2 \pi \prod_{i=1}^{q}\left( \lambda _{i}^{2}+x^{2}\right) ^{p_{i}}
	x ^{2H-1-2p}}{\sigma ^{2}\Gamma \left( 2H+1\right) \sin \left(
	H\pi \right) }\leq k \prod_{i=1}^{q}\left( \widetilde{\lambda }%
_{i}^{2}+x^{2}\right) ^{p_{i}} x ^{2H-1-2p}. \label{cota1/f}
\end{equation}%
Using that $H\in \left[ h_1,h_2\right] $ there exists $h' \in [h_1,h_2]$
such that (\ref{cota1/f}) is less than or
equal to 
\begin{equation*}
k \prod_{i=1}^{q}\left( \widetilde{\lambda }%
_{i}^{2}+x^{2}\right) ^{p_{i}} x^{2h'-1-2p} :=g(x).
\end{equation*}%
Therefore

\begin{equation*}
A_{n,T}\left( \lambda \right) \leq M_{n,T}\frac{T}{n}\sum_{j=1}^{n}\frac{w\left(
	jT/n\right) }{f^{\left( X\right) }\left( jT/n,\lambda ,\widehat{\sigma },%
	\widehat{H}\right) }\leq M_{n,T}\frac{T}{n}\sum_{j=1}^{n}w\left( jT/n\right)
g\left( jT/n\right) .
\end{equation*}%
Observe that $\frac{T_n}{n}\sum_{j=1}^{n}w\left( jT_n/n\right) g\left(
jT_n/n\right) \overset{a.s.}{\rightarrow }\int_{0}^{+\infty}w(x)g(x)dx < + \infty$ and using
Lemma \ref{supI_T}, we have $M_{n,T_n}\overset{P}{\rightarrow }0$ we obtain that $%
A_{n,T_n}=\sup_{\lambda \in \Lambda }A_{n,T_n}\left( \lambda \right) \overset{P}{%
	\rightarrow }0.$
	\end{proof}

\begin{proof}[Proof of Theorem \ref{consistency_theorem}]\[\]
Fix $T>0.$ Given $\varepsilon >0,$ because the minimum of 
$U_{T}\left( \lambda ,\sigma
^{0},H^{0}\right) $ is reached in a unique point $\widehat{\lambda }_{T},$
 we deduce that there exists a
random variable $\delta_T >0$ such that the condition $U_{T}\left( \lambda
,\sigma ^{0},H^{0}\right) <U_{T}\left( \widehat{\lambda }_{T},\sigma
^{0},H^{0}\right) +\delta_T $ implies that $\left\vert \widehat{\lambda }%
_{T}-\lambda \right\vert <\varepsilon .$ 
Theorem 3.6 of Kalemkerian \& Le\'{o}n (\cite{chichi}) shows that $\widehat{\lambda }_{T}%
\overset{P}{\rightarrow }\lambda ^{0}$ as $T\rightarrow +\infty ,$ then $%
\widehat{\lambda }_{T_{n}}\overset{P}{\rightarrow }\lambda ^{0}$ as $%
n\rightarrow +\infty .$ Also $\widehat{\lambda}_T$ is reached in a unique point, then   
given $\varepsilon ,\varepsilon ^{\prime }>0$ there exists $T_{0}>0$ and a random
variable $\delta _{T_{0}}>0$ such that $P\left( \left\vert \widehat{\lambda }%
_{T}-\lambda ^{0}\right\vert \leq \varepsilon /2\right) \geq 1-\varepsilon
^{\prime }/2$  for all $T\geq T_{0}$ and 
\begin{equation}
\left \{ U_{T_{0}}\left( \lambda ,\sigma ^{0},H^{0}\right) < U_{T_{0}}\left( 
\widehat{\lambda }_{T_{0}},\sigma ^{0},H^{0}\right) +\delta _{T_{0}} \right \} \subset
 \left \{ \left\vert \lambda -\widehat{\lambda }_{T_{0}}\right\vert \leq
\varepsilon /2 \right \}. \label{105}
\end{equation}

Also, $%
\widehat{\lambda }_{T_{n}}\overset{P}{\rightarrow }\lambda ^{0}$  it follows that exist $n_{0}$ such that for any $%
n\geq n_{0}$ the following conditions are fulfilled: 
\begin{equation}
P\left( A_{n}\right) =P\left( \left\vert \widehat{\lambda }_{T_{n}}-\lambda
^{0}\right\vert \leq \varepsilon /2\right) \geq 1-\varepsilon ^{\prime }/4,%
\text{ }\label{An}
\end{equation}%
\begin{equation}
P\left( B_{n}\right) =P\left( \left\vert U_{T_{0}}\left( \widehat{\lambda }%
_{T_{0}},\sigma ^{0},H^{0}\right) -U_{T_{0}}\left( \widehat{\lambda }%
_{T_{n}},\sigma ^{0},H^{0}\right) \right\vert \leq \delta _{T_{0}}/5\right)
\geq 1-\varepsilon ^{\prime }/4, \label{Bn}
\end{equation}%
\begin{equation}
P\left( C_{n}\right) =P\left( \sup_{\lambda \in \Lambda }\left\vert U_{T_0}\left(
\lambda ,\sigma ^{0},H^{0}\right) -U_{T_{n}}\left( \lambda ,\sigma
^{0},H^{0}\right) \right\vert \leq \delta _{T_{0}}/5\right) \geq
1-\varepsilon ^{\prime }/4, \label{Cn}
\end{equation}%
\begin{equation}
P\left( D_{n}\right) =P\left( \sup_{\lambda \in \Lambda }\left\vert
U_{T_{n}}^{\left( n\right) }\left( \lambda ,\widehat{\sigma },\widehat{H}%
\right) -U_{T_{n}}\left( \lambda ,\sigma ^{0},H^{0}\right) \right\vert
\leq \delta _{T_{0}}/5 \right) \geq 1-\varepsilon ^{\prime }/4.\text{ \ } \label{Dn}
\end{equation}%
 (\ref{Bn}) it follows from 
continuity of $U_{T_0}$, (\ref{Cn}) from Lemma \ref{lemmaU_T} and (\ref{Dn}) from 
Lemma 4.\\
Suppose that $A_n \cap B_n \cap C_n \cap D_n$ occurs.

On the one hand, from (\ref{Dn}) and (\ref{lambdasgorro}), the definition of 
$\widehat{\lambda}_{T_n}^{(n)}$, we obtain that \ \ 
\begin{equation*}
U_{T_{n}}\left( \widehat{\lambda }_{T_{n}}^{\left( n\right) },\sigma
^{0},H^{0}\right) \leq U_{T_{n}}^{\left( n\right) }\left( \widehat{\lambda }%
_{T_{n}}^{\left( n\right) },\widehat{\sigma },\widehat{H}\right) +\delta
_{T_{0}}/5\leq 
\end{equation*}
$$U_{T_{n}}^{\left( n\right) }\left( \widehat{\lambda }_{T_{n}},%
\widehat{\sigma },\widehat{H}\right) +\delta _{T_{0}}/5\leq U_{T_{n}}\left( 
\widehat{\lambda }_{T_{n}},\sigma ^{0},H^{0}\right) +2\delta _{T_{0}}/5.$$

Then 
\begin{equation}
U_{T_{n}}\left( \widehat{\lambda }_{T_{n}}^{\left( n\right) },\sigma
^{0},H^{0}\right) \leq U_{T_{n}}\left( \widehat{\lambda }_{T_{n}},\sigma
^{0},H^{0}\right) +2\delta _{T_{0}}/5. \label{111}
\end{equation}%
On the other hand, from (\ref{Cn}) and (\ref{111}) it follows that 
\begin{equation}
U_{T_{0}}\left( \widehat{\lambda }_{T_{n}}^{\left( n\right) },\sigma
^{0},H^{0}\right) \leq U_{T_{n}}\left( \widehat{\lambda }_{T_{n}}^{\left(
	n\right) },\sigma ^{0},H^{0}\right) +\delta _{T_{0}}/5
	\leq U_{T_n}\left ( \widehat{\lambda}_{T_n},\sigma^{0},H^{0} \right )+
	3\delta_{T_0}/5. \label{102}
\end{equation}%
 Also, (\ref{Cn}) and (\ref{Bn}) implies that 
\begin{equation}
U_{T_{n}}\left( \widehat{\lambda }_{T_{n}},\sigma ^{0},H^{0}\right) \leq
U_{T_{0}}\left( \widehat{\lambda }_{T_{n}},\sigma ^{0},H^{0}\right) +\delta
_{T_{0}}/5\leq \label{103}
\end{equation}%
\begin{equation}
U_{T_{0}}\left( \widehat{\lambda }_{T_{0}},\sigma
^{0},H^{0}\right) +2\delta _{T_{0}}/5. \label{104}
\end{equation}

From (\ref{102}), (\ref{103}) and (\ref{104}) we obtain 
\begin{equation*}
U_{T_{0}}\left( \widehat{\lambda }_{T_{n}}^{\left( n\right) },\sigma
^{0},H^{0}\right) \leq U_{T_{0}}\left( \widehat{\lambda }_{T_{0}},\sigma
^{0},H^{0}\right) +\delta _{T_{0}}
\end{equation*}%
\ \ \ \ and from (\ref{105}), we obtain $\left\vert \widehat{\lambda }%
_{T_{n}}^{\left( n\right) }-\widehat{\lambda }_{T_{0}}\right\vert $ $\leq
\varepsilon /2,$ thus $\left\vert \widehat{\lambda }_{T_{n}}^{\left(
	n\right) }-\lambda ^{0}\right\vert $ $\leq \varepsilon .$

Then, we have shown that $A_{n}\cap B_{n}\cap C_{n}\cap D_{n}\subset \left\{
\left\vert \widehat{\lambda }_{T_{n}}^{\left( n\right) }-\lambda
^{0}\right\vert \leq \varepsilon \right\} $, therefore $P\left( \left\vert 
\widehat{\lambda }_{T_{n}}^{\left( n\right) }-\lambda ^{0}\right\vert \leq
\varepsilon \right) \geq 1-\varepsilon ^{\prime }$ for all $n\geq n_{0}.$

\end{proof}

\begin{proof}[Proof of Proposition 1]\[\]
	Formula (\ref{FOUl1l2}) is given in \cite{chichi}.\\
	(\ref{FOUl1l1}) it is immediately by taking limit $\beta \rightarrow 0$ in (\ref{FOUl1l1l2}).

\end{proof}

\begin{proof}[Proof of Proposition 2]
	\[\] 
	The formula (\ref{FOUl1l2l3}) is given in \cite{chichi}.\\
	To obtain (\ref{FOUl1l1l2}), it is enough to take $\lim_{\gamma \rightarrow \alpha }
	$ in (\ref{FOUl1l2l3}). Firstly, to simplify the calculation, we put $a=\alpha ^{2},$ $b=\beta
	^{2}$ and $c=\gamma ^{2}.$ Then (\ref{FOUl1l2l3}) becomes 
	
	\begin{equation}
	\mathbb{E}\left( X_{0}X_{t}\right) =\frac{\sigma ^{2}H}{2}\left[ \frac{%
		a^{2-H}f_{H}\left( \sqrt{a}t\right) }{\left( a-b\right) \left( a-c\right) }+%
	\frac{b^{2-H}f_{H}\left( \sqrt{b}t\right) }{\left( b-a\right) \left(
		b-c\right) }+\frac{c^{2-H}f_{H}\left( \sqrt{c}t\right) }{\left( c-a\right)
		\left( c-b\right) }\right] .\label{fouabc}
	\end{equation}%
	Thus, when $c\rightarrow a$ (\ref{fouabc}) this becomes 
	\begin{equation*}
	\frac{\sigma ^{2}H}{2}\left[ \frac{b^{2-H}f_{H}\left( \sqrt{b}t\right) }{%
		\left( b-a\right) ^{2}}+\lim_{c\rightarrow a}\frac{a^{2-H}f_{H}\left( \sqrt{a%
		}t\right) \left( c-b\right) -c^{2-H}f_{H}\left( \sqrt{c}t\right) \left(
		a-b\right) }{\left( a-b\right) \left( a-c\right) \left( c-b\right) }\right] =
	\end{equation*}%
	\begin{equation}
	\frac{\sigma ^{2}H}{2}\left[ \frac{b^{2-H}f_{H}\left( \sqrt{b}t\right) }{%
		\left( b-a\right) ^{2}}+\frac{1}{\left( b-a\right) ^{2}}\lim_{c\rightarrow a}%
	\frac{a^{2-H}f_{H}\left( \sqrt{a}t\right) \left( c-b\right)
		-c^{2-H}f_{H}\left( \sqrt{c}t\right) \left( a-b\right) }{a-c}\right] . \label{fouabcc}
	\end{equation}%
	Applying L'H\^{o}pital rule we obtain that (\ref{fouabcc}) is equal to 
	$\frac{\sigma ^{2}H}{2\left( b-a\right) ^{2}} \times $
	\begin{equation*}
	 b^{2-H}f_{H}\left( \sqrt{b}t\right) -\lim_{c\rightarrow
		a} \left[ a^{2-H}f_{H}\left( \sqrt{a}t\right) -\left( \left( 2-H\right)
	c^{1-H}f_{H}\left( \sqrt{c}t\right) +\frac{c^{3/2-H}t}{2}f_{H}^{\prime
	}\left( \sqrt{c}t\right) \right) \left( a-b\right) \right] =
	\end{equation*}
	
	\begin{equation}
	\frac{\sigma ^{2}H}{2}\left[ \frac{b^{2-H}f_{H}\left( \sqrt{b}t\right)
		-a^{2-H}f_{H}\left( \sqrt{a}t\right) +\left( \left( 2-H\right)
		a^{1-H}f_{H}\left( \sqrt{a}t\right) +\frac{a^{3/2-H}t}{2}f_{H}^{\prime
		}\left( \sqrt{a}t\right) \right) \left( a-b\right) }{\left( b-a\right) ^{2}}%
	\right] . \label{fouaab}
	\end{equation}%
	Lastly, replacing $a=\alpha ^{2}$ and $b=\beta ^{2}$, we obtain (\ref{FOUl1l1l2}).
	
	This concludes the proof of (\ref{FOUl1l1l2}).
		
	To prove (\ref{FOUl1l1l1}) as was done in the previous formula, it is enough to take limit where $b\rightarrow a$ in (\ref{fouaab}).
	Then applying L'H\^{o}pital's rule we obtain that 
	\begin{equation*}
	\frac{\sigma ^{2}H}{2}\left[ \lim_{b\rightarrow a}\frac{b^{2-H}f_{H}\left( 
		\sqrt{b}t\right) -a^{2-H}f_{H}\left( \sqrt{a}t\right) +\left( \left(
		2-H\right) a^{1-H}f_{H}\left( \sqrt{a}t\right) +\frac{a^{3/2-H}t}{2}%
		f_{H}^{\prime }\left( \sqrt{a}t\right) \right) \left( a-b\right) }{\left(
		b-a\right) ^{2}}\right] =
	\end{equation*}%
	\begin{equation}
	\frac{\sigma ^{2}H}{4}\left[ \lim_{b\rightarrow a}\frac{\left( 2-H\right)
		b^{1-H}f_{H}\left( \sqrt{b}t\right) +\frac{b^{3/2-H}t}{2}f_{H}\left( \sqrt{b}%
		t\right) -\left( 2-H\right) a^{1-H}f_{H}\left( \sqrt{a}t\right) -\frac{%
			a^{3/2-H}t}{2}f_{H}^{\prime }\left( \sqrt{a}t\right) }{b-a}\right] . \label{fouaaa}
	\end{equation}
	
	Applying again L'H\^{o}pital's rule, (\ref{fouaaa}) is equal to 
	\begin{equation*}
	\frac{\sigma ^{2}H}{4}\left[ \left( 2-H\right) \left( 1-H\right)
	a^{-H}f_{H}\left( \sqrt{a}t\right) +\left( 7-4H\right) \frac{a^{1/2-H}t}{2}%
	f_{H}^{\prime }\left( \sqrt{a}t\right) +\frac{a^{1-H}t^{2}}{4}f_{H}^{\prime
		\prime }\left( \sqrt{a}t\right) \right] .
	\end{equation*}%
	Lastly, putting $a=\alpha ^{2}$ we obtain (\ref{FOUl1l1l1}). 
\end{proof}

\end{document}